\documentclass[11pt]{amsart}   % includes 11pt font size (default 10pt)

\usepackage{verbatim}
\usepackage{amssymb,amsthm,amsmath}
\usepackage{color}
\definecolor{darkblue}{rgb}{0, 0, .4}
\definecolor{grey}{rgb}{.7, .7, .7}

\usepackage{ifpdf}
\ifpdf
  \usepackage[pdftex]{epsfig}

  \usepackage{lscape}

  \usepackage[breaklinks]{hyperref}
  \hypersetup{
            colorlinks=true,
            linkcolor=darkblue,
            anchorcolor=darkblue,
            citecolor=darkblue,
            pagecolor=darkblue,
            urlcolor=darkblue,
            pdftitle={},
            pdfauthor={}
  }
\else
  \usepackage[dvips]{epsfig}
  \usepackage{lscape}

  \newcommand{\href}[2]{#2}
  \newcommand{\url}[2]{#2}
\fi

\newtheorem{theorem}{Theorem}[section]
\newtheorem{lemma}[theorem]{Lemma}

\theoremstyle{definition}
\newtheorem{definition}[theorem]{Definition}
\newtheorem{example}[theorem]{Example}

\theoremstyle{remark}
\newtheorem{remark}[theorem]{Remark}

\numberwithin{equation}{section}

\theoremstyle{theorem}
\newtheorem{corollary}[theorem]{Corollary}

\newcommand{\f}[0]{\varphi}

\usepackage{graphicx}
\input xy 
\xyoption{all}

\usepackage{tikz}
\usetikzlibrary{decorations.pathreplacing}
\usetikzlibrary{trees}
\usepackage{ifthen}

\renewcommand{\S}{\mathcal{S}}
\newcommand{\SN}{\mathfrak{S}}
\renewcommand{\S}{\mathfrak{S}}

\newcommand{\pf}[1]{\pi|_{[#1]}}

\renewcommand{\P}[0]{\mathcal{P}}

\usetikzlibrary{patterns}

\renewcommand{\a}{{\centerdot}}
\newcommand{\BTM}{\mathrm{BTM}}
\newcommand{\BTMW}{\mathrm{BTM_{WORD}}}
\newcommand{\BTMP}{\mathrm{BTM_{PATH}}}
\newcommand{\pgeq}{\succcurlyeq}

\begin{document}

\title{Reselection in the game of best choice}

\begin{abstract}
We investigate a remarkable probability distribution on the symmetric
group, due to Steck from the early 1970's, arising from a natural process that
intertwines continuous and discrete selections for the values and positions,
respectively, of a permutation.  Steck used a matrix determinant to express his
distribution, whereas we contribute new combinatorial formulas for it in terms
of ``bottom-to-top maxima'' (that are simply the left-to-right maxima of the
inverse) permutation statistics.  These formulas specialize, in the case of the
identity permutation, to a result of Pitman--Stanley from the late 1990's.

\medskip \noindent
We then use the Steck distribution to define a game of best choice (secretary
problem variation) that incorporates a filtering process for the pool of
candidates over time.  We solve the model for the case where there is a single
filtering step.  It turns out that the probability of winning the game under
optimal play is similar to the classical model ($1/e$, asymptotically), but
that the interviewer must employ a different (non-positional) strategy in order
to attain it.  The optimal strategy depends on the relationship between
interview position and the next bottom-to-top maximum value after the filtering
step.  Among other results, we prove that the optimal strategy always
transitions from rejection to acceptance between positions $(1/e)$ and $(1/e) +
(1 - 1/e)y$ as a proportion of the total candidates considered, where $y$ is
the proportion of unfiltered candidates.
\end{abstract}

\author{Dillon Hanson}
\address{Department of Mathematics, Jacksonville University,
Jacksonville, Florida 32211, USA}
\author[Brant Jones]{Brant Jones}
\address{Department of Mathematics and Statistics, MSC 1911, James Madison University, Harrisonburg, VA 22807}
\email{\href{mailto:jones3bc@jmu.edu}{\texttt{jones3bc@jmu.edu}}}
\urladdr{\url{http://educ.jmu.edu/\~jones3bc/}}
\author{Hyejin Kim} %\ \ (University of Michigan-Dearborn)\\
\address{Department of Mathematics and Statistics, University of Michigan--Dearborn, Dearborn, MI 48128, USA}
%\keywords{}

\date{\today}

\maketitle

%%%%%%%%%%%%%%%%%%%%%%%%%%%%%%%%%%%%%%%%%%%%%%%%%%%%%%%%%%%%%%%%%%%%%
%  Section
%%%%%%%%%%%%%%%%%%%%%%%%%%%%%%%%%%%%%%%%%%%%%%%%%%%%%%%%%%%%%%%%%%%%%
\bigskip
\section{Introduction}

We begin with a question:  Why does the classical game of best choice (also
known as the ``secretary problem'') assume a uniform distribution for the
candidate rankings?  To answer this, consider the game as an elementary
discrete probability model with the schematic shown in Figure~\ref{f:scheme}.
There is a domain of candidates, illustrated in blue, that is finite but much
larger than the number $N$ of candidates we are willing to interview.  Among
applications that have been suggested in popular media, the domain could be
houses within a given price range in our zip code, potential romantic partners
in our city, or particular items that we are interested in purchasing through
an auction website.  At each time step $t = 1, \ldots, N$, we ``conduct an
interview'' by sampling the domain (uniformly) and {\bf projecting} that sample
to a permutation of size $t$ using the relative ranks of the candidates
interviewed so far as the values.  Most interesting domains probably have
various ``coordinates'' potentially relevant to the interviewer's ranking
(e.g., in the hiring story: years of education, details of prior experience,
salary history, etc.) that are fully revealed only during an interview.  In the
classical model, this sample-and-project process is simply repeated $N$ times
by an interviewer that is assumed to have no control over the static domain,
and so the model acquires in this way a {\em uniform distribution} on $\S_N$,
the permutations of size $N$ that represent all possible relative rankings.
This is the point from which classical analysis ensues, leading to the famous
$1/e$ result.  See \cite{gilbert--mosteller,freeman,ferguson} for surveys of
this original problem and related results.  The essence of this problem also
arises naturally in theoretical computer science and auction theory, and
continues to attract attention (see e.g.
\cite{milenkovic,pinsky,buchbinder--etal,kleinberg08}).

\begin{figure}[t]
\[ \scalebox{1.1}{ \begin{tikzpicture}
\draw [pattern=north west lines, pattern color=blue, minimum size=4.5mm] (0,10) -- (3,11) -- (10,11) -- (7,10) -- (0,10);
\draw (2.5,10.4) node [circle, fill=blue, inner sep=0pt, minimum size=4pt, draw] {} ;
\draw [dotted] (2.5, 10.4) -- (4.4,8);
\draw (11.3,10.5) node {\parbox{1.2in}{\tiny (finite) domain of candidates \\ in (high-dimensional) ``reality''}};
\draw [->] (11,9.8) -- (11,8.5);
\draw (9.5, 9.2) node {\parbox{0.9in}{\tiny sample and {\it project}:  compare latest sample to each prior sample}};
\draw [->,color=black!60!green] (11.2,8.5) -- (11.2,9.8);
\draw [color=black!60!green] (12.8, 9.2) node {\parbox{0.95in}{\tiny refine
    ranking criteria and {\it reselect} domain:  improve candidate pool}};  
\draw (11.3,8) node {\parbox{1.0in}{\tiny permutations of size $N$ \\(relative rankings / time)}};
\draw (4.4,8) node [circle, fill=black, inner sep=0pt, minimum size=3pt, draw] {} ;
\draw (4.7,8) node [circle, fill=black, inner sep=0pt, minimum size=3pt, draw] {} ;
\draw (5.0,8) node [circle, fill=black, inner sep=0pt, minimum size=3pt, draw] {} ;
\draw (5.7,8) node [circle, fill=black, inner sep=0pt, minimum size=3pt, draw] {} ;
\draw [->] (4,8) -- (7,8);
\draw (5.5,7.5) node {\tiny time};
\draw [pattern=north east lines, pattern color=black!60!green] (5.7,10.4) circle [x radius=1.1,y radius=0.2];
\draw (6.3,10.4) node [circle, fill=black!60!green, inner sep=0pt, minimum size=4pt, draw] {} ;
\draw [dotted] (6.3, 10.4) -- (5.7,8);
\draw [color=gray] (5.7,8) .. controls (5.35,7.8) .. (5.0,8);
\draw [color=gray] (5.7,8) .. controls (5.2,7.7) .. (4.7,8);
\draw [color=gray] (5.7,8) .. controls (5.05,7.6) .. (4.4,8);
\draw [thick, dotted, color=black!60!green] (4.6,10.4) -- (5.7,8);
\draw [thick, dotted, color=black!60!green] (6.8,10.4) -- (5.7,8);
\draw [thick, dotted, color=blue] (0,10) -- (4.4,8);
\draw [thick, dotted, color=blue] (3,11) -- (4.4,8);
\draw [thick, dotted, color=blue] (10,11) -- (4.4,8);
\draw [thick, dotted, color=blue] (7,10) -- (4.4,8);
\draw [color=black!60!green] (4.6,10.4) -- (5.7,8);
\draw [color=black!60!green] (6.8,10.4) -- (5.7,8);
\draw [color=blue] (0,10) -- (4.4,8);
\draw [color=blue] (3,11) -- (4.4,8);
\draw [color=blue] (10,11) -- (4.4,8);
\draw [color=blue] (7,10) -- (4.4,8);
\end{tikzpicture} }\]
\caption{Origin of the symmetric group distribution in the game of best choice} \label{f:scheme}
\end{figure}
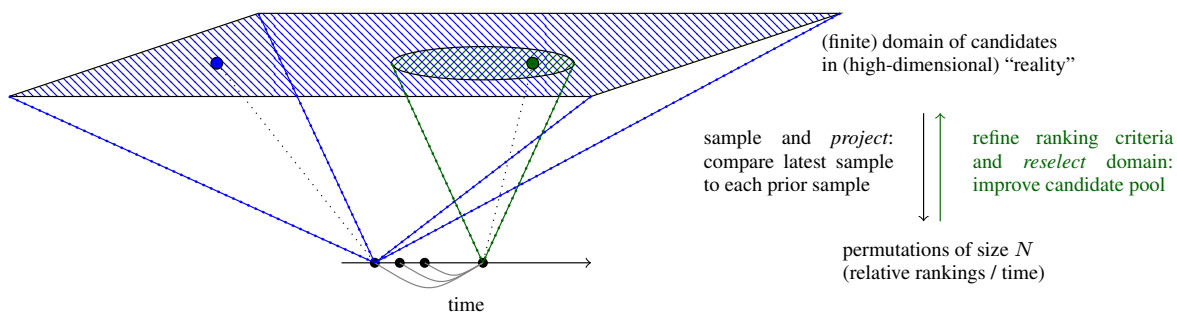

As we ponder the origin of this uniform distribution, however, as a sequence of
projections from a high-dimensional domain over time, this setup begins to seem
remarkably unstable.  In \cite{weighted}, we have already found some
quantitative examples where the classical game arises {\em discontinuously} as
a limit of some natural families of non-uniform weighted models.  In
particular, any extrinsic change in the domain of candidates (e.g.  candidates
entering or exiting the domain or changes arising from trends in general market
rates or prices) can significantly change the projected distribution of
relative rankings, even if we continue to sample each domain uniformly at each
time step.  Said another way, the rationale for assuming a uniform distribution
on the permutations seems to stem from the reasonable goal of modeling a
uniform distribution of candidates {\em in space}; however, this assumption
simultaneously requires the candidates to be uniformly distributed {\em in
time}, which is a much less realistic assumption.  Among other things, it
implies that the interviewer cannot have any influence over the candidate pool
(because otherwise, they would try to improve it, resulting in permutations
that tend to be increasing rather than uniform).  It also means that the
candidate pool should remain exactly the same over the entire course of the $N$
interviews, a process that may take months or years in some of the popular
applications and may be unrealistic even for theoretical computer science
applications of the problem.

In this paper, we develop an alternative model based on the idea that an
interviewer can and does change the domain intrinsically as they conduct
interviews.  To capture this, we propose extending the model to include a
feedback loop as illustrated in green in Figure~\ref{f:scheme}.  Our
justification for this stems from the observation that in many domains, the
interviewer must make {\em increasingly complex comparisons} among the
candidates they have seen.  Although a cursory evaluation may suffice to
compare two candidates, the process of precisely ranking ten initial
candidates, say, as the algorithm demands, becomes progressively more involved.
The criteria used to rank candidates at the beginning of the game may not be
sufficient to distinguish individual ranks as the list of interviewed
candidates grows.  To make these finer rank comparisons, the interviewer must
develop new criteria which can only come from a deeper understanding of the
domain.  Therefore, {\em some domain learning process for the interviewer seems
inherent in the ranking process required at each step by the classical
algorithm}.  As the interviewer applies this learning process to continue to be
able to completely rank candidates at each step, they should simultaneously be
using this information to hone the process for obtaining candidates in the
first place and continuously improve the pool under consideration.  In our
model, we refer to this honing process as the {\bf reselect} step.

In Section~\ref{s:2} we introduce an underlying distribution on the symmetric
group that enables reselection to occur.  Although we derived this
independently, motivated by our philosophical investigations of the classical
game, it turns out that this distribution has already appeared before in work
of Steck \cite{steck} and been specialized for study in a polytopal setting by
Pitman--Stanley \cite{pitman--stanley}.  Our development relating this
distribution to the ``bottom-to-top maxima'' (that are simply dual to the usual
``left-to-right maxima'') permutation statistics appears to be new.
The Pitman--Stanley polytope remains an object of active research in combinatorial
geometry (e.g. \cite{meszaros--morales}) and our results suggest new directions
for generalization.

Analyzing the reselect games in full generality seems technically difficult at
present, so starting in Section~\ref{s:3}, we specialize to the case where there is a
single reselection step, occurring at some position $I$ among the $N$
interviews.  The strength of the reselect step is also a parameter of the
model.  In Section~\ref{s:3}, we show that the optimal strategy only depends on
the relationship between the number of candidates interviewed so far and the
value of the next ``bottom-to-top maximum'' rank occurring after the
reselection step.  We find a nonrecursive formula for the probability of
winning the game when accepting various candidate rankings.  To find the
optimal strategy, we then need to compare with the corresponding formula for
the probability of winning the game when rejecting a given candidate ranking.
In Section~\ref{s:4}, we obtain recursive formulas for these, and provide
nonrecursive solutions of them in Section~\ref{s:5}.  Finally, we are able to
solve the one-step reselect models explicitly in Section~\ref{s:6} (see
Theorems~\ref{t:themain} and \ref{t:themainp}) and demonstrate some example
computations.

%%%%%%%%%%%%%%%%%%%%%%%%%%%%%%%%%%%%%%%%%%%%%%%%%%%%%%%%%%%%%%%%%%%%%
%  Section
%%%%%%%%%%%%%%%%%%%%%%%%%%%%%%%%%%%%%%%%%%%%%%%%%%%%%%%%%%%%%%%%%%%%%
\bigskip
\section{The Reselect Distribution}\label{s:2}

Fix a positive integer $N$ and let $\S_N$ denote the symmetric group of rank
$N$.  Consider a random experiment in which $N$ points are selected from each
of $N$ columns.  More precisely, the height of the point in each column is a
{\em continuous} variable selected uniformly from the interval $(0, 1)$; each
of the $N$ columns is {\em discrete} with the same size and alignment as shown
in the left side of Figure~\ref{f:pex}.  Then, we obtain a diagram of $N$
points whose {\bf relative order} (where we interpret positions horizontally
and values vertically) gives rise to a uniformly random permutation of $N$.
For example, the left point configuration in Figure~\ref{f:pex} has relative
order $15324$.

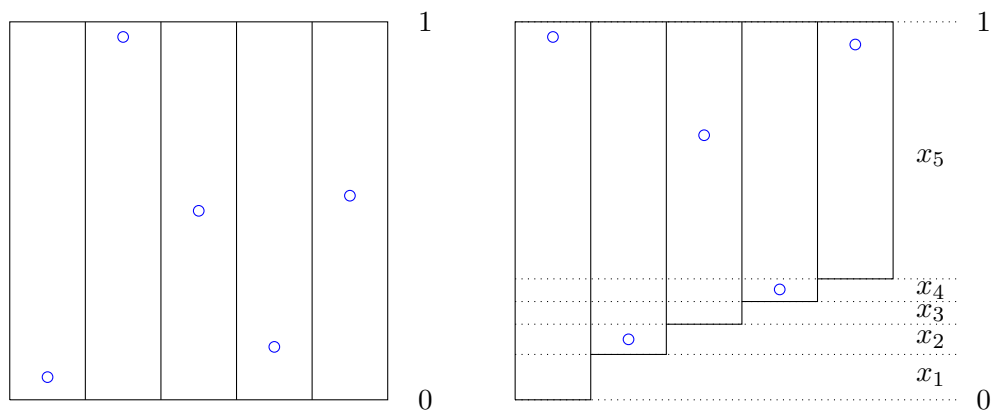
\begin{figure}[h]
\[ \scalebox{1.0}{ 
\begin{tikzpicture}[grow=down]
	\draw (0,0) -- (5,0) -- (5,5) -- (0,5) -- (0,0);
	\draw (1,0) -- (1,5);
	\draw (2,0) -- (2,5);
	\draw (3,0) -- (3,5);
	\draw (4,0) -- (4,5);
	\draw (5.5,0) node {$0$};
	\draw (5.5,5) node {$1$};
	\draw[blue] (.5,0.3) circle (2pt);
	\draw[blue] (1.5,4.8) circle (2pt);
	\draw[blue] (2.5,2.5) circle (2pt);
	\draw[blue] (3.5,0.7) circle (2pt);
	\draw[blue] (4.5,2.7) circle (2pt);
    %\node at (2.5,-0.5) {(a)};
\end{tikzpicture}
	\hspace{0.3in}
\begin{tikzpicture}[grow=down]
	\draw (0,0) -- (1,0) -- (1,0.6) -- (2,0.6) -- (2,1.0) -- (3,1.0) -- (3,1.3) -- (4,1.3) -- (4,1.6) -- (5,1.6) -- (5,5) -- (0,5) -- (0,0);
	\draw (1,0.6) -- (1,5);
	\draw (2,1.0) -- (2,5);
	\draw (3,1.3) -- (3,5);
	\draw (4,1.6) -- (4,5);
	\draw (6.2,0) node {$0$};
	\draw (6.2,5) node {$1$};
	\draw [dotted] (0,0) -- (5.9,0);
	\draw [dotted] (0,0.6) -- (5.9,0.6);
	\draw [dotted] (0,1.0) -- (5.9,1.0);
	\draw [dotted] (0,1.3) -- (5.9,1.3);
	\draw [dotted] (0,1.6) -- (5.9,1.6);
	\draw [dotted] (0,5) -- (5.9,5);
	%\draw (5.5,0) node {$0$};
	\draw (5.5,0.3) node {$x_1$};
	\draw (5.5,0.8) node {$x_2$};
	\draw (5.5,1.15) node {$x_3$};
	\draw (5.5,1.45) node {$x_4$};
	\draw (5.5,3.2) node {$x_5$};
	%\draw (5.5,5) node {$1$};
	\draw[blue] (0.5,4.8) circle (2pt);
	\draw[blue] (1.5,0.8) circle (2pt);
	\draw[blue] (2.5,3.5) circle (2pt);
	\draw[blue] (3.5,1.46) circle (2pt);
	\draw[blue] (4.5,4.7) circle (2pt);
    %\node at (2.5,-0.5) {(b)};
\end{tikzpicture}
}
\]
	\caption{The reselect experiment produces permutations $15324$ and $51324$, respectively}\label{f:pex}
\end{figure}

Next, suppose we modify the experiment such that the heights of the columns are
still aligned at the top of the diagram but their sizes are restricted as we
move to the right, as shown in Figure~\ref{f:pex}.  Specifically, let $x_i$
denote the difference between the bottom of the $(i+1)$st and $i$ columns, with
$x_N$ equal to the size of the last column.  The height of the selected point
in column $i$ is still uniformly distributed on the interval of length $x_i +
x_{i+1} + \cdots + x_N$ that ends at $1$.  In this case, the relative orderings of
the $N$ points, considered as a permutation from $\S_N$, will not have a
uniform distribution; among other things, there will be some bias towards
permutations that are increasing.

\begin{definition}
Fix parameters $x_1, x_2, \ldots, x_{N-1}$ from the interval $[0, 1)$ such that $x_N := 1-\sum_{i=1}^{N-1} x_i$ is positive and consider the probability $R(\pi)$ of obtaining each $\pi \in \S_N$ as a result of the modified experiment described above.  We refer to this as the {\bf reselect distribution}, or for reasons that will become clear shortly, the {\bf Steck distribution} on $\S_N$ with parameters $x_i$.
\end{definition}

Notice that when $x_i = 0$ for all $1 \leq i < N$ with $x_N = 1$, we recover
the uniform distribution on $\S_N$.  In order to find a formula for the
reselect distribution in general, we require some definitions.

First, observe that each column in the reselect experiment is divided into
regions by the heights of the all the subsequent columns; see the dotted lines
in Figure~\ref{f:pex}.  We refer to any region of a column whose size is
$x_i$ as {\bf region $i$}, regardless of which column it appears in.  Given a
configuration of points from the reselect experiment, we define its {\bf
sequence of regions} to be $(r_1, r_2, \ldots, r_N)$ where $r_i$ is the region
of the $i$th largest point from the collection, reading vertically from
smallest to largest.  For example, $(2, 4, 5, 5, 5)$ is the sequence of regions
for the points on the right side of Figure~\ref{f:pex}.  By definition, these sequences must
be weakly increasing.

Next, we fix some permutation $\pi \in \S_N$ and consider the possible
sequences of regions that could conceivably give rise to $\pi$ (as the relative
ordering of a point configuration from the reselect experiment).  To refer to
individual entries of a permutation, we say that $\pi(i)$ is the {\bf value} at
{\bf position} $i$, so $\pi^{-1}(j)$ denotes the {\em position} of value $j$ in
$\pi$.

\begin{definition}\label{d:vra}
Given $\pi \in \S_N$, we say that $(r_1, r_2, \cdots, r_N)$ is a {\bf valid} sequence of regions for $\pi$ if 
\begin{itemize}
    \item each $r_i$ is between $\pi^{-1}(i)$ (that is, the position of value $i$ in $\pi$) and $N$, and
    \item the sequence of $r_i$ are increasing.
\end{itemize}
\end{definition}

In our example, the valid region sequences for $\pi = 51324$ must have 
\[ r_1 \in \{2, 3, 4, 5\}, \ \ \ \ r_2, r_3 \in \{4, 5\}, \ \ \ \ r_4 = r_5 = 5, \ \ \ \text{ with } \ \ \ r_1 \leq r_2 \leq r_3 \leq r_4 \leq r_5, \]
and we observe that the sequence of regions $(2, 4, 5, 5, 5)$ shown in
Figure~\ref{f:pex} does conform to these inequalities.

Now it is straightforward to verify that if $\pi$ arises as the relative
ordering of some point configuration from the reselect experiment, then the
sequence of regions for those points must satisfy the conditions in
Definition~\ref{d:vra}.  Moreover, all point configurations having a sequence
of regions that are valid for $\pi$ (as per Definition~\ref{d:vra}) {\em and}
having the correct relative ordering for points from regions that are repeated,
will produce $\pi$ as the relative ordering.

Next, we describe the minimal valid sequence of regions for a given $\pi \in \S_N$.

\begin{definition}
Given $\pi \in \S_N$, we say that a {\bf bottom-to-top maximum} of $\pi$ is an entry whose {\em position} is to the right of all entries having smaller {\em value}.
\end{definition}
We remark that the (perhaps more familiar) left-to-right maxima of $\pi^{-1}$ form the bottom-to-top maxima of $\pi$ (since inversion reverses the roles of positions and values in a permutation).  For example, the bottom-to-top maxima for each permutation in $\S_3$ are indicated below using dots:
\[ \dot1\dot2\dot3,\ \ \ \ \dot13\dot2,\ \ \ \ 2\dot1\dot3,\ \ \ \ 3\dot1\dot2,\ \ \ \ 23\dot1,\ \ \ \ 32\dot1.  \]
Observe that the entries with value $1$ or position $N$ are always bottom-to-top maxima.

\begin{definition}\label{d:minvr}
Given $\pi \in \S_N$, define $\BTMW(\pi) = (m_1, m_2, \ldots, m_N)$ as follows.
Let $m_1$ be the position $\pi^{-1}(1)$ of the value $1$ in $\pi$.  Then for each $i = 2, 3, \ldots, N$, let
\[ m_i = \begin{cases}
            \pi^{-1}(i) & \text{ if $i$ is a bottom-to-top maximum in $\pi$, } \\
            m_{i-1} & \text{ otherwise. }
\end{cases} \]
\end{definition}

For example, $\pi = 51324$ produces $\BTMW(\pi) = (2, 4, 4, 5, 5)$.  We claim this is the {\em minimal} valid sequence of regions for $\pi$.

\begin{lemma}\label{l:vr}
We have that $(r_1, r_2, \ldots, r_N)$ is a valid sequence of regions for $\pi$
if and only if the $r_i$ are an increasing integer sequence such that $m_i \leq
r_i \leq N$ where $(m_1, m_2, \ldots, m_N) = \BTMW(\pi)$.
\end{lemma}
\begin{proof}
Given a permutation $\pi$, read the values of $\pi$ in ascending order.  The
value $1$ occurs in some position $\pi^{-1}(1) = m_1$ and this is the minimal
possible value for $r_1$, by definition. 

We argue by induction.  Once we have determined the minimal values for $r_1,
r_2, \ldots, r_{i-1}$, we can find the minimal value for $r_i$ as follows.  If
value $i$ occurs to the left of the largest bottom-to-top maxima among the
values $\{1, 2, \ldots, i-1\}$, then $r_i$ must be at least as large as
$r_{i-1}$ in order to be increasing in Definition~\ref{d:vra}.  By our
induction hypothesis we have $r_{i-1} \geq m_{i-1}$, and $m_{i-1} = m_i$ by
Definition~\ref{d:minvr}.

Otherwise, the value $i$ occurs to the right of the last bottom-to-top maximum,
so forms a new bottom-to-top maximum itself.  In this case, $r_i$ must be
greater than or equal to than the position $\pi^{-1}(i)$ of $i$ by the first
condition in Definition~\ref{d:vra}, which is equal to $m_i$ by
Definition~\ref{d:minvr}.

Thus, the stated logical equivalence holds for $r_1, r_2, \ldots, r_i$, so the
result follows by induction.
\end{proof}

We can visualize these definitions graphically, using {\bf Dyck paths} in the
Cartesian plane; by definition these are lattice paths running from $(0,0)$ to
$(N,N)$, consisting of $(0,1)$ steps (i.e.  north) and $(1,0)$ steps (i.e.
east), always remaining above the line $y = x$.  The {\bf northwest corners} in
a Dyck path consist of a north step immediately followed by an east step.  We
label each northwest corner by the row and column at the {\em end} of its
east step; this is the same as the natural labeling for the cell in the $N
\times N$ array bounded by the northwest corner (see illustrations below).

\begin{definition}\label{d:btm}
Let $\BTMP(\pi)$ be the lattice path from $(0,0)$ to $(N,N)$ whose height in
column $i$ is given by the $i$th entry of $\BTMW(\pi)$, with vertical steps as
necessary to create a connected path.
\end{definition}

It is straightforward to verify that $\BTMP(\pi)$ always lies above the line $y
= x$, so is a Dyck path.  Moreover, the bottom-to-top maxima of $\pi$
correspond precisely to the northwest corners of $\BTMP(\pi)$.  For example, we
have the following $\BTMP(\pi)$ and $\BTMW(\pi)$ corresponding to the possible
sets of bottom-to-top maxima that arise in $\S_3$ shown in
Figure~\ref{f:s3btm}.

\begin{figure}[h]
\[
\begin{array}{r@{\quad}l}
\pi = 
&
\hspace{0.7in}\dot1\dot2\dot3,\hspace{.6in}
\dot13\dot2,\hspace{.5in}
2\dot1\dot3,\hspace{.5in}
3\dot1\dot2,\hspace{.45in}
\underbrace{23\dot1,\ \ \ \ 32\dot1}
\\[1.0em]
\mathrm{BTM}_{\mathrm{PATH}}(\pi) =
&
\vcenter{\hbox{
\scalebox{0.4}{ 
\begin{tikzpicture}
\node at (-20,0) { \begin{tikzpicture}
        \foreach \i in {0,...,3} { \draw [dashed] (\i,\i) -- (\i,3); }
        \foreach \i in {0,...,3} { \draw [dashed] (0,\i) -- (\i,\i); }
        \draw [dashed] (0,0) -- (3,3);
        \draw [line width=2pt] (0,0) node [circle, fill=black, inner sep=0pt, minimum size=5pt, draw] {} 
        -- ++(0,1) node [circle, fill=black, inner sep=0pt, minimum size=5pt, draw] {}
        -- ++(1,0) node [circle, fill=black, inner sep=0pt, minimum size=5pt, draw] {}
        -- ++(0,1) node [circle, fill=black, inner sep=0pt, minimum size=5pt, draw] {}
        -- ++(1,0) node [circle, fill=black, inner sep=0pt, minimum size=5pt, draw] {}
        -- ++(0,1) node [circle, fill=black, inner sep=0pt, minimum size=5pt, draw] {}
        -- ++(1,0) node [circle, fill=black, inner sep=0pt, minimum size=5pt, draw] {};
    \draw (0.5,-0.5) node {\huge $1$};
    \draw (1.5,-0.5) node {\huge $2$};
    \draw (2.5,-0.5) node {\huge $3$};
    \draw (-0.5,0.5) node {\huge $1$};
    \draw (-0.5,1.5) node {\huge $2$};
    \draw (-0.5,2.5) node {\huge $3$};
    \draw (1.5,-1.2) node {\huge values};
    \draw (-2.5,1.5) node {\huge positions};
  \end{tikzpicture} };

\node at (-13,0) { \begin{tikzpicture}
        \foreach \i in {0,...,3} { \draw [dashed] (\i,\i) -- (\i,3); }
        \foreach \i in {0,...,3} { \draw [dashed] (0,\i) -- (\i,\i); }
        \draw [dashed] (0,0) -- (3,3);
        \draw [line width=2pt] (0,0) node [circle, fill=black, inner sep=0pt, minimum size=5pt, draw] {} 
        -- ++(0,1) node [circle, fill=black, inner sep=0pt, minimum size=5pt, draw] {}
        -- ++(1,0) node [circle, fill=black, inner sep=0pt, minimum size=5pt, draw] {}
        -- ++(0,1) node [circle, fill=black, inner sep=0pt, minimum size=5pt, draw] {}
        -- ++(0,1) node [circle, fill=black, inner sep=0pt, minimum size=5pt, draw] {}
        -- ++(1,0) node [circle, fill=black, inner sep=0pt, minimum size=5pt, draw] {}
        -- ++(1,0) node [circle, fill=black, inner sep=0pt, minimum size=5pt, draw] {};
    \draw (0.5,-0.5) node {\huge $1$};
    \draw (1.5,-0.5) node {\huge $2$};
    \draw (-0.5,0.5) node {\huge $1$};
    \draw (-0.5,2.5) node {\huge $3$};
    \draw (1.5,-1.2) node {\huge values};
  \end{tikzpicture} };

\node at (-8,0) { \begin{tikzpicture}
        \foreach \i in {0,...,3} { \draw [dashed] (\i,\i) -- (\i,3); }
        \foreach \i in {0,...,3} { \draw [dashed] (0,\i) -- (\i,\i); }
        \draw [dashed] (0,0) -- (3,3);
        \draw [line width=2pt] (0,0) node [circle, fill=black, inner sep=0pt, minimum size=5pt, draw] {} 
        -- ++(0,1) node [circle, fill=black, inner sep=0pt, minimum size=5pt, draw] {}
        -- ++(0,1) node [circle, fill=black, inner sep=0pt, minimum size=5pt, draw] {}
        -- ++(1,0) node [circle, fill=black, inner sep=0pt, minimum size=5pt, draw] {}
        -- ++(1,0) node [circle, fill=black, inner sep=0pt, minimum size=5pt, draw] {}
        -- ++(0,1) node [circle, fill=black, inner sep=0pt, minimum size=5pt, draw] {}
        -- ++(1,0) node [circle, fill=black, inner sep=0pt, minimum size=5pt, draw] {};
    \draw (0.5,-0.5) node {\huge $1$};
    \draw (2.5,-0.5) node {\huge $3$};
    \draw (-0.5,1.5) node {\huge $2$};
    \draw (-0.5,2.5) node {\huge $3$};
    \draw (1.5,-1.2) node {\huge values};
  \end{tikzpicture} };

\node at (-3,0) { \begin{tikzpicture}
        \foreach \i in {0,...,3} { \draw [dashed] (\i,\i) -- (\i,3); }
        \foreach \i in {0,...,3} { \draw [dashed] (0,\i) -- (\i,\i); }
        \draw [dashed] (0,0) -- (3,3);
        \draw [line width=2pt] (0,0) node [circle, fill=black, inner sep=0pt, minimum size=5pt, draw] {} 
        -- ++(0,1) node [circle, fill=black, inner sep=0pt, minimum size=5pt, draw] {}
        -- ++(0,1) node [circle, fill=black, inner sep=0pt, minimum size=5pt, draw] {}
        -- ++(1,0) node [circle, fill=black, inner sep=0pt, minimum size=5pt, draw] {}
        -- ++(0,1) node [circle, fill=black, inner sep=0pt, minimum size=5pt, draw] {}
        -- ++(1,0) node [circle, fill=black, inner sep=0pt, minimum size=5pt, draw] {}
        -- ++(1,0) node [circle, fill=black, inner sep=0pt, minimum size=5pt, draw] {};
    \draw (0.5,-0.5) node {\huge $1$};
    \draw (1.5,-0.5) node {\huge $2$};
    \draw (-0.5,1.5) node {\huge $2$};
    \draw (-0.5,2.5) node {\huge $3$};
    \draw (1.5,-1.2) node {\huge values};
  \end{tikzpicture} };

\node at (2,0) { \begin{tikzpicture}
        \foreach \i in {0,...,3} { \draw [dashed] (\i,\i) -- (\i,3); }
        \foreach \i in {0,...,3} { \draw [dashed] (0,\i) -- (\i,\i); }
        \draw [dashed] (0,0) -- (3,3);
        \draw [line width=2pt] (0,0) node [circle, fill=black, inner sep=0pt, minimum size=5pt, draw] {} 
        -- ++(0,1) node [circle, fill=black, inner sep=0pt, minimum size=5pt, draw] {}
        -- ++(0,1) node [circle, fill=black, inner sep=0pt, minimum size=5pt, draw] {}
        -- ++(0,1) node [circle, fill=black, inner sep=0pt, minimum size=5pt, draw] {}
        -- ++(1,0) node [circle, fill=black, inner sep=0pt, minimum size=5pt, draw] {}
        -- ++(1,0) node [circle, fill=black, inner sep=0pt, minimum size=5pt, draw] {}
        -- ++(1,0) node [circle, fill=black, inner sep=0pt, minimum size=5pt, draw] {};
    \draw (0.5,-0.5) node {\huge $1$};
    \draw (-0.5,2.5) node {\huge $3$};
    \draw (1.5,-1.2) node {\huge values};
  \end{tikzpicture} };
\end{tikzpicture}
}}}
\\[3.0em]
\mathrm{BTM}_{\mathrm{WORD}}(\pi) =
&
\hspace{0.7in}(1, 2, 3),\hspace{.25in}
(1, 3, 3),\hspace{.25in}
(2, 2, 3),\hspace{.25in}
(2, 3, 3),\hspace{.25in}
(3, 3, 3)
\end{array}
\]
\caption{For each $\pi \in \mathfrak{S}_3$, we illustrate $\mathrm{BTM}_{\mathrm{PATH}}(\pi)$ and $\mathrm{BTM}_{\mathrm{WORD}}(\pi)$.}
\label{f:s3btm}
\end{figure}
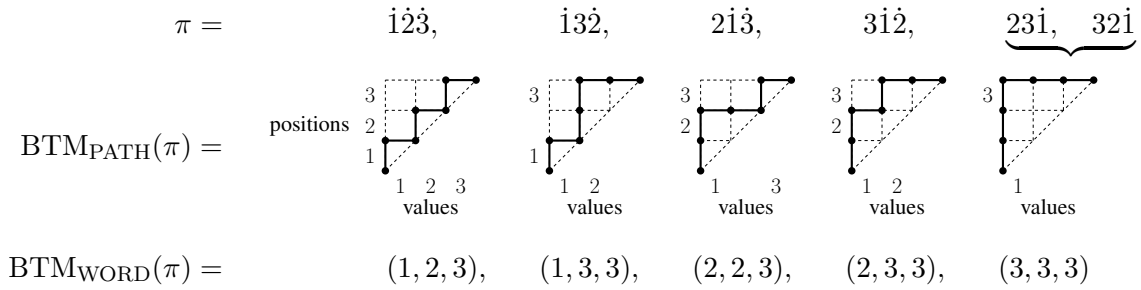

We are now in position to state our main results.

\begin{theorem}\label{t:rmain}
In the reselect distribution with parameters $x_1, x_2, \ldots, x_N$, the relative ordering $\pi \in \S_N$ occurs with probability
\[ R(\pi) = \frac{1}{\prod_{i=1}^N (x_i + x_{i+1} + \cdots + x_N)} \sum_{\substack{ \text{Dyck paths }\\ p \pgeq \mathrm{BTM_{PATH}}(\pi)}} \frac{ \prod_{i=1}^N x_{p_i} }{\prod_{j=1}^N p^{(j)}!} \]
\[ = \frac{1}{N! \ \prod_{i=1}^N (x_i + x_{i+1} + \cdots + x_N)} \sum_{\substack{ \text{integer sequences }\\ w \unrhd \mathrm{BTM_{WORD}}(\pi)}} \prod_{i=1}^N x_{w_i} \]
where
\begin{itemize}
\item for two Dyck paths $p$ and $q$, we say that $p \pgeq q$ if the shape of the Dyck path $p$ lies weakly above the shape of $q$,
\item for two integer sequences $w$ and $v$, we say that $w \unrhd v$ if the non-decreasing rearrangement of $w$, denoted $(w_{(1)}, \ldots, w_{(N)})$, satisfies $N \geq w_{(i)} \geq v_i$ for each $i$,
\item $p_i = $ height of the horizontal step in the $i$th column of the Dyck path $p$, and
\item $p^{(j)} = $ length of the horizontal segment in the $j$th row of the Dyck path $p$. % (so in particular, $\sum_{j=1}^N p^{(j)} = N$).
\end{itemize}
\end{theorem}

\begin{proof}
We can compute $R(\pi)$, the probability of obtaining $\pi$ as the
relative ordering in the reselect experiment, by summing the probabilities of
obtaining each valid sequence of regions and the correct relative ordering for
points from regions that are repeated.

The valid sequences of regions are given by Dyck paths $p \pgeq \BTMP(\pi)$
according to Lemma~\ref{l:vr} and Definition~\ref{d:btm}.  These are disjoint
events, so we may sum their probabilities.

In the bijection between Dyck paths and region assignments, the height $p_i$ of
the horizontal step in the $i$th column of $p$ represents the region of the
$i$th value (sorted from smallest to largest) which should then occur in some
column $j = \pi^{-1}(i)$ of the point configuration.  Since each column has a
uniform distribution, the probability of this event in the reselect experiment
is therefore $\frac{x_{p_i}}{x_j + x_{j+1} + \cdots + x_N}$.  As the point in
each column of the configuration is selected independently, the probability
that all of them are selected as assigned by $p$ is the product of these
formulas.  The denominators do not depend on $p$ so we collect them together
and pull them out of the sum.

Finally, if a set of $k$ distinct values from $\pi$ are all assigned to the
same region, any of the $k!$ relative orderings will be equally likely in the
reselect experiment.  Thus, we further divide by $\prod_{j=1}^N p^{(j)}!$ to
find the probability that the point configuration has the same relative order
prescribed by $\pi$ for these points assigned to repeated regions.

This proves the first equality.  Each Dyck path $p$ in the first sum
corresponds to a non-decreasing word containing $p^{(j)}$ copies of $j$.  Thus,
each term from the first sum corresponds to $\frac{N!}{\prod_{j=1}^N p^{(j)}!}$
permuted integer sequences (each having the same content) in the second sum by
the Multinomial Theorem.  This yields the second equality.
\end{proof}

\begin{remark}
The integer sequences appearing in the second formula for $R(\pi)$ are
essentially {\bf vector parking functions}, defined with non-standard
conventions to match the parameters of our reselect experiment.  See \cite{yan,
pitman--stanley} for more about these objects.
\end{remark}

\begin{remark}
In the reselect distribution, permutations having the same set of (values and
positions of) bottom-to-top maxima have the same probability of occurrence.
In particular, the permutations whose last entry is $1$ all have the same
probability of occurrence.  This is due to the fact that all of these
permutations must have their entries selected from region $N$ in every column.
\end{remark}

It turns out that the probabilities $R(\pi)$ can also be obtained from a
determinantal formula of Steck (see \cite{steck,pitman}), which explains our
alternative nomenclature for this distribution.  Consider the
following specialization of \cite[Theorem 1]{pitman} where all the upper limits
are $1$ and we transpose the matrix indices.

\begin{theorem}{\bf (Steck)}
Let $0 \leq \ell_1 \leq \ell_2 \leq \cdots \leq \ell_N \leq 1$ be given and suppose that $q_1, q_2, \ldots, q_N$ are the ``order statistics'' (in ascending order) from a sample of $N$ independent uniform random variables with a range from $0$ to $1$.  Then, the probability that $\ell_i \leq q_i$ is 
\[ N! \ \det\left[  (1-\ell_i)^{i-j+1} / (i-j+1)!  \right]_{1 \leq i,j \leq N} \]
where we interpret any matrix entry with $i-j+1<0$ as zero (where the factorial is undefined).
\end{theorem}

\begin{corollary}\label{c:rmain}
Given $\pi \in \S_N$, let $\BTMW(\pi) = (m_1, m_2, \ldots, m_N)$.
Then, 
\[ R(\pi) =  \frac{1}{\prod_{i=1}^N (x_i + x_{i+1} + \cdots + x_N)} \ \det\left[  (x_{m_i} + x_{m_i+1} + \cdots + x_N)^{i-j+1} / (i-j+1)!  \right]_{1 \leq i,j \leq N}. \]
\end{corollary}
\begin{proof}
    Let $y = (y_1, y_2, \ldots, y_N)$ be a sample of $N$ independent uniform random variables with a range from $0$ to $1$.  Say that $y$ is {\bf in the reselect region} if each $y_i \geq 1 - (x_i + x_{i+1} + \cdots + x_N)$.

Then, $R(\pi)$ is the conditional probability 
\[ \P(y \text{ has relative order $\pi$} \ |\  y \text{ is in the reselect region}). \]
This is
\[ \frac{\P(y \text{ has relative order $\pi$} \land y \text{ is in the reselect region})}{\prod_{i=1}^N (x_i + x_{i+1} + \cdots + x_N)} \]
since the product in the denominator is the probability that $y$ is in the reselect region (the point in each column being drawn independently).

By Lemma~\ref{l:vr} from our development, we have that $y$ will have relative order $\pi$ if and only if the region of the $i$th largest point of $y$ (reading increasingly) is at least $m_i$.  But this is precisely the same as requiring the ``order statistic'' $q_i$ to be at least $x_1+x_2+\cdots+x_{m_{i}-1}$ (see Figure~\ref{f:pex}).  Thus, we may apply Steck's Theorem with $1 - \ell_i = x_{m_i} + x_{m_i + 1} + \cdots + x_N$.

Since $y$ satisfying Steck's event is independent of the particular relative ordering $\pi$ we assumed for $y$, we multiply Steck's formula by $1/N!$, obtaining our stated result.
\end{proof}

Although this determinantal formula is not as amenable to combinatorial proof as the other formulas we found, it is fast, elegant, and has the advantage of clearly expressing the inequalities that define the event whose probability is being computed.  Putting our results together, we can generalize a result from Pitman--Stanley~\cite{pitman--stanley}.  Equations (2) and (5) from their Theorem 1 correspond to the single choice $m = (1, 2, \ldots, N)$ which is $\BTMW(12\cdots N)$ in our distribution on the symmetric group, but as our Theorem~\ref{t:rmain} and Corollary~\ref{c:rmain} each compute $R(\pi)$, we find that equality holds for arbitrary $m$.

\begin{corollary}
Fix a non-decreasing integer sequence $m = (m_1, m_2, \ldots, m_N)$ such that $i \leq m_i \leq N$ for each $i$.
Then, 
\[ \det\left[  (x_{m_i} + x_{m_i+1} + \cdots + x_N)^{i-j+1} / (i-j+1)!  \right]_{1 \leq i,j \leq N} = \sum_{w \in K_N(m)} \frac{x_1^{w_1} x_2^{w_2} \cdots x_N^{w_N} }{w_1! w_2! \cdots w_N!} \]
where we interpret any matrix entry with $i-j+1<0$ as zero and
\[ K_N(m) = \{w \in \mathbb{N}^N : \sum_{i=0}^{j-1} w_{N-i} \geq \#\{i:m_i\geq N-j+1\} \text{ for all $1 \leq j \leq N-1$ and } \sum_{i=1}^N w_i = N \}. \]
\end{corollary}
Presumably, other results from the Pitman--Stanley paper also generalize for arbitrary choices of $m = \BTMW(\pi)$ but we do not pursue this here.

\begin{example}
\begin{figure}[ht]
\scalebox{0.4}{ \begin{tikzpicture}
\node at (0,0) (A) { \begin{tikzpicture}
        \foreach \i in {0,...,3} { \draw [dashed] (\i,\i) -- (\i,3); }
        \foreach \i in {0,...,3} { \draw [dashed] (0,\i) -- (\i,\i); }
        \draw [dashed] (0,0) -- (3,3);
		\draw [line width=2pt] (0,0) node [circle, fill=black, inner sep=0pt, minimum size=5pt, draw] {} 
		-- ++(0,1) node [circle, fill=black, inner sep=0pt, minimum size=5pt, draw] {}  % UP
		-- ++(0,1) node [circle, fill=black, inner sep=0pt, minimum size=5pt, draw] {}  % UP
		-- ++(0,1) node [circle, fill=black, inner sep=0pt, minimum size=5pt, draw] {}  % UP
		-- ++(1,0) node [circle, fill=black, inner sep=0pt, minimum size=5pt, draw] {}  % RIGHT
		-- ++(1,0) node [circle, fill=black, inner sep=0pt, minimum size=5pt, draw] {}  % RIGHT
		-- ++(1,0) node [circle, fill=black, inner sep=0pt, minimum size=5pt, draw] {}  % RIGHT
		;
    \draw (1.5,4.2) node {\huge $23\dot1$, $32\dot1$};    
	\draw (1,-1) node {\huge $x_3^3 / 3!$};
  \end{tikzpicture} };
\node at (8,0) (B) { \begin{tikzpicture}
        \foreach \i in {0,...,3} { \draw [dashed] (\i,\i) -- (\i,3); }
        \foreach \i in {0,...,3} { \draw [dashed] (0,\i) -- (\i,\i); }
        \draw [dashed] (0,0) -- (3,3);
		\draw [line width=2pt] (0,0) node [circle, fill=black, inner sep=0pt, minimum size=5pt, draw] {} 
		-- ++(0,1) node [circle, fill=black, inner sep=0pt, minimum size=5pt, draw] {}  % UP
		-- ++(0,1) node [circle, fill=black, inner sep=0pt, minimum size=5pt, draw] {}  % UP
		-- ++(1,0) node [circle, fill=black, inner sep=0pt, minimum size=5pt, draw] {}  % RIGHT
		-- ++(0,1) node [circle, fill=black, inner sep=0pt, minimum size=5pt, draw] {}  % UP
		-- ++(1,0) node [circle, fill=black, inner sep=0pt, minimum size=5pt, draw] {}  % RIGHT
		-- ++(1,0) node [circle, fill=black, inner sep=0pt, minimum size=5pt, draw] {}  % RIGHT
		;
    \draw (1.5,4.2) node {\huge $3\dot1\dot2$};    
	\draw (1,-1) node {\huge $x_2 x_3^2 / 2!$};
  \end{tikzpicture} };
\node at (16,-4) (C) { \begin{tikzpicture}
        \foreach \i in {0,...,3} { \draw [dashed] (\i,\i) -- (\i,3); }
        \foreach \i in {0,...,3} { \draw [dashed] (0,\i) -- (\i,\i); }
        \draw [dashed] (0,0) -- (3,3);
		\draw [line width=2pt] (0,0) node [circle, fill=black, inner sep=0pt, minimum size=5pt, draw] {} 
		-- ++(0,1) node [circle, fill=black, inner sep=0pt, minimum size=5pt, draw] {}  % UP
		-- ++(0,1) node [circle, fill=black, inner sep=0pt, minimum size=5pt, draw] {}  % UP
		-- ++(1,0) node [circle, fill=black, inner sep=0pt, minimum size=5pt, draw] {}  % RIGHT
		-- ++(1,0) node [circle, fill=black, inner sep=0pt, minimum size=5pt, draw] {}  % RIGHT
		-- ++(0,1) node [circle, fill=black, inner sep=0pt, minimum size=5pt, draw] {}  % UP
		-- ++(1,0) node [circle, fill=black, inner sep=0pt, minimum size=5pt, draw] {}  % RIGHT
		;
    \draw (1.5,4.2) node {\huge $\dot1 3\dot2$};    
	\draw (1,-1) node {\huge $x_2^2 x_3 / 2!$};
  \end{tikzpicture} };
\node at (16,4) (D) { \begin{tikzpicture}
        \foreach \i in {0,...,3} { \draw [dashed] (\i,\i) -- (\i,3); }
        \foreach \i in {0,...,3} { \draw [dashed] (0,\i) -- (\i,\i); }
        \draw [dashed] (0,0) -- (3,3);
		\draw [line width=2pt] (0,0) node [circle, fill=black, inner sep=0pt, minimum size=5pt, draw] {} 
		-- ++(0,1) node [circle, fill=black, inner sep=0pt, minimum size=5pt, draw] {}  % UP
		-- ++(1,0) node [circle, fill=black, inner sep=0pt, minimum size=5pt, draw] {}  % RIGHT
		-- ++(0,1) node [circle, fill=black, inner sep=0pt, minimum size=5pt, draw] {}  % UP
		-- ++(0,1) node [circle, fill=black, inner sep=0pt, minimum size=5pt, draw] {}  % UP
		-- ++(1,0) node [circle, fill=black, inner sep=0pt, minimum size=5pt, draw] {}  % RIGHT
		-- ++(1,0) node [circle, fill=black, inner sep=0pt, minimum size=5pt, draw] {}  % RIGHT
		;
    \draw (1.5,4.2) node {\huge $2\dot1\dot3$};    
	\draw (1,-1) node {\huge $x_1 x_3^2 / 2!$};
  \end{tikzpicture} };
\node at (24,0) (E) { \begin{tikzpicture}
        \foreach \i in {0,...,3} { \draw [dashed] (\i,\i) -- (\i,3); }
        \foreach \i in {0,...,3} { \draw [dashed] (0,\i) -- (\i,\i); }
        \draw [dashed] (0,0) -- (3,3);
		\draw [line width=2pt] (0,0) node [circle, fill=black, inner sep=0pt, minimum size=5pt, draw] {} 
		-- ++(0,1) node [circle, fill=black, inner sep=0pt, minimum size=5pt, draw] {}  % UP
		-- ++(1,0) node [circle, fill=black, inner sep=0pt, minimum size=5pt, draw] {}  % RIGHT
		-- ++(0,1) node [circle, fill=black, inner sep=0pt, minimum size=5pt, draw] {}  % UP
		-- ++(1,0) node [circle, fill=black, inner sep=0pt, minimum size=5pt, draw] {}  % RIGHT
		-- ++(0,1) node [circle, fill=black, inner sep=0pt, minimum size=5pt, draw] {}  % UP
		-- ++(1,0) node [circle, fill=black, inner sep=0pt, minimum size=5pt, draw] {}  % RIGHT
		;
    \draw (1.5,4.2) node {\huge $\dot1\dot2\dot3$};    
	\draw (1,-1) node {\huge $x_1 x_2 x_3 / 1!$};
  \end{tikzpicture} };
\draw (A) -- (B);
\draw (B) -- (C);
\draw (B) -- (D);
\draw (C) -- (E);
\draw (D) -- (E);
    \end{tikzpicture}}
\caption{Contributions of $\mathrm{BTM}_{\mathrm{PATH}}(\pi)$ to the combinatorial $R(\pi)$ formula}
\label{f:s3}
\end{figure}
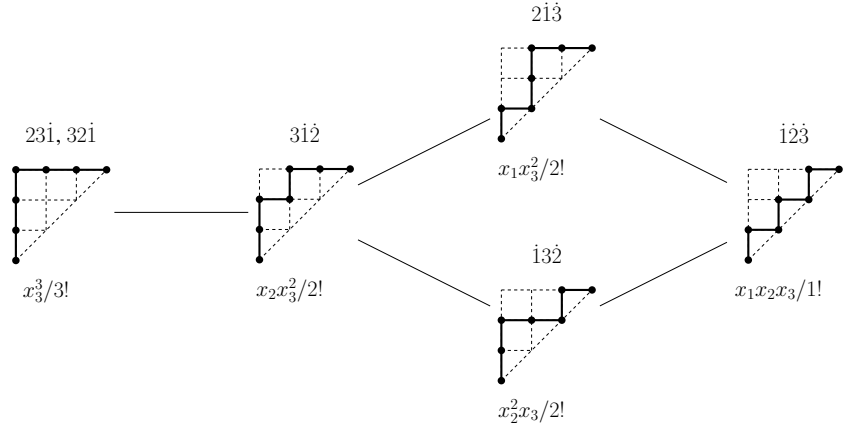

For $N = 3$ there are five Dyck paths.  Their $\pgeq$ relationship and contributions to the sum are shown in Figure~\ref{f:s3}.
Thus, we have the following distribution (omitting the normalizing constant):
\[ R( 123 ) = \frac{x_3^3}{3!} + \frac{x_2 x_3^2}{2!} + \frac{x_2^2 x_3}{2!} + \frac{x_1 x_3^2}{2!} + \frac{x_1 x_2 x_3}{1!},
\hspace{.7in} R( 1\a2 ) = \frac{x_3^3}{3!} + \frac{x_2 x_3^2}{2!}  + \frac{x_1 x_3^2}{2!},  \]
\[ R( \a13 ) = \frac{x_3^3}{3!} + \frac{x_2 x_3^2}{2!} + \frac{x_2^2 x_3}{2!},   
\hspace{.4in} R( \a12 ) = \frac{x_3^3}{3!} + \frac{x_2 x_3^2}{2!},   
\hspace{.4in} R( \a\a1 ) = \frac{x_3^3}{3!}.   \]
Remarkably, the sum of these $R(123) + R(1\a2) + R(\a13) + R(\a12) + 2 R(\a\a1)$ (over the full symmetric group) equals the normalizing constant
\[ (x_1+x_2+x_3)(x_2+x_3)x_3 = x_1 x_2 x_3 + x_1 x_3^2 + x_2^2 x_3 + 2 x_2 x_3^2 + x_3^3, \]
a natural refinement of the Mahonian generating function, in order to have the probabilities sum to 1.
\end{example}

%%%%%%%%%%%%%%%%%%%%%%%%%%%%%%%%%%%%%%%%%%%%%%%%%%%%%%%%%%%%%%%%%%%%%
%  Section
%%%%%%%%%%%%%%%%%%%%%%%%%%%%%%%%%%%%%%%%%%%%%%%%%%%%%%%%%%%%%%%%%%%%%
\bigskip
\section{The one-step reselect model}\label{s:3}

%%%%%%%%%%%%%%%%%%%%%%%%%%%%%%%%%%%%%%%%%%%%%%%%%%%%%%%%%%%%%%%%%%%%%
%  Subsection
%%%%%%%%%%%%%%%%%%%%%%%%%%%%%%%%%%%%%%%%%%%%%%%%%%%%%%%%%%%%%%%%%%%%%
\bigskip
\subsection{The reselect distribution}

Here, we specialize to models where there is only a single reselection step, occurring at position $I$.  That is, for fixed $N$, $1 \leq I < N$, and $x \in [0, 1)$, we define the parameters for the reselect distribution to be
\begin{equation*}
x_i = \begin{cases}
x, & \text{if } i=I \\
1-x, & \text{if } i=N\\
0, & \text{ otherwise}.
\end{cases}
\end{equation*}
Just as $R(\pi)$ is constant on permutations having the same positions and values of bottom-to-top maxima by Theorem~\ref{t:rmain}, the specialization of $R(\pi)$ in the one-step reselect models will turn out to be constant on permutations having the same ``$J$-statistic'' that we introduce below.

\begin{definition}
For $\pi \in \S_{N}$ and $I < N$, the \textit{J-statistic} of $\pi$, denoted $J(\pi)$, is the value of the next bottom-to-top maximum that appears after position $I$ in $\pi$.
\end{definition}

Although we suppress it from the notation for simplicity, observe that the $J$-statistic depends on $I$ as well as $\pi$.  Notice that the $J$-statistic can be at most $I+1$, and this value is attained precisely when $\{\pi_1,\dots, \pi_{I}\} = \{1,\dots, I\}$.

\begin{example}
Fix $\pi=\dot1\dot2\dot356\dot4\dot7\in \S_7$. When $I=3$, $J(\pi)=4$, whereas, when $I=6$, $J(\pi)=7$.
\end{example}

\begin{theorem}\label{t:1main}
Suppose $N$ and $0 \leq I < N$ are fixed.  For $J = 1, 2, \ldots, I+1$, define 
\[ R_J(x) := \frac{1}{N!\ (1-x)^{N-I}} \sum_{h=0}^{J-1} {N \choose h} x^h(1-x)^{N-h}. \]
Then $R(\pi) = R_J(x)$ for all $\pi \in \S_N$ with $J(\pi)=J$.
\end{theorem}
\begin{proof}
First, recall Theorem~\ref{t:rmain}.  The sum over Dyck paths in the general formula 
    \[ R(\pi) = \frac{1}{\prod_{i=1}^N (x_i + x_{i+1} + \cdots + x_N)} \sum_{\substack{ \text{Dyck paths }\\ w \succcurlyeq \mathrm{BTM}(\pi)}} \frac{ \prod_{i=1}^N x_{w_i} }{\prod_{j=1}^N w^{(j)}!} \]
simplifies considerably as many terms are zeroed out in our specialization, where $x_I = x$ and $x_N = (1-x)$ are the only nonzero parameters.  In particular, the only Dyck paths that will contribute nonzero summands are those whose horizontal steps occur {\em only} at levels $I$ and $N$.  There are $I+1$ such Dyck paths, and we denote by $D_h$ the Dyck path from $(0,0)$ to $(N,N)$ that has $h$ horizontal steps on level $I$ followed by $N-h$ horizontal steps on level $N$, for $0 \leq h \leq I$.  Then for $h \neq 0$, we have that $D_h$ has northwest corners at $(1,I)$ and $(h+1,N)$ and nowhere else; for $h = 0$, we have exactly one northwest corner at $(1,N)$ for $D_0$. 

Since $J(\pi) = J$, the Dyck path $\BTM(\pi)$ from Definition~\ref{d:btm} has a northwest corner at $(J, m)$ for some $m > I$.  Since $J$ is the first bottom-to-top maximum after position $I$, any bottom-to-top maximum that precedes it would have a height of at most $I$.  Therefore, $D_h \succcurlyeq \BTM(\pi)$ if and only if $h < J$.  Thus, we obtain 
\begin{align*}
  R(\pi) &= \frac{1}{\prod_{i=1}^N (x_i + x_{i+1} + \cdots + x_N)} \sum_{\substack{ \text{Dyck paths }\\ w \succcurlyeq \mathrm{BTM}(\pi)}} \frac{ \prod_{i=1}^N x_{w_i} }{\prod_{j=1}^N w^{(j)}!}\\
  &= \frac{1}{(1-x)^{N-I}} \sum_{\substack{ \text{Dyck paths } w = D_h\\ \text{with } h < J} } \frac{ \prod_{i=1}^N x_{w_i} }{\prod_{j=1}^N w^{(j)}!}\\
         &= \frac{1}{(1-x)^{N-I}} \sum_{h=0}^{J-1} \frac{x^h(1-x)^{N-h}}{h!(N-h)!}\\
\end{align*}
which is equivalent to the expression $R_J(x)$ given in the statement.
\end{proof}

These $R_J(x)$ are polynomials of degree $I$ with rational coefficients, and we recover the classical uniform distribution in the cases where $I = 0$ or $x = 0$, where we interpret $0^0$ as $1$.

%%%%%%%%%%%%%%%%%%%%%%%%%%%%%%%%%%%%%%%%%%%%%%%%%%%%%%%%%%%%%%%%%%%%%
%  Subsection
%%%%%%%%%%%%%%%%%%%%%%%%%%%%%%%%%%%%%%%%%%%%%%%%%%%%%%%%%%%%%%%%%%%%%
\bigskip
\subsection{The strike probabilities}\label{s:strike}

In the remainder of this paper, we consider the game of best choice played on a
collection of interview orderings that are distributed according to the
$R_J(x)$.  We recommend \cite{weighted} or \cite{jones18} for a gentle
introduction to games of best choice from the combinatorial perspective.  

\begin{definition}\label{d:pflat}
Given a sequence of $i$ distinct integers, we define its {\bf flattening} to be
the unique permutation of $\{1, 2, \ldots, i\}$ having the same relative order
as the elements of the sequence.  Given a permutation $\pi$, define the {\bf
$i$th prefix flattening}, denoted $\pf{i}$, to be the permutation obtained by
flattening the sequence $\pi_1, \pi_2, \ldots, \pi_i$ of values appearing in
the first $i$ positions of $\pi$.  We say $\pi$ {\bf has $p$ as a prefix} if
some prefix flattening $\pf{i}$ is equal to $p$.

In the {\bf game of best choice} (with parameters $N$ and $I$) some $\pi \in
\SN_N$ is chosen randomly (here, with probability $R(\pi)$ as computed in
Theorem~\ref{t:1main}) and each prefix flattening $\pf{1}, \pf{2}, \ldots$ is
presented sequentially to the player.  For each prefix flattening, the player
chooses whether to {\bf accept} the current candidate (and stop playing) or
{\bf reject} the current candidate (and continue playing).  If the player stops
at value $N$, they {\bf win}; otherwise, they {\bf lose}.
\end{definition}

We are interested in determining the optimal strategies and probability of
winning the game.  For a prefix $p \in \bigcup_{i=1}^N \SN_i$, say that $\pi
\in \SN_N$ is {\bf $p$-acceptable} if one of its prefix flattenings is $p$ and
accepting the prefix flattening $p$ would win the game with interview order
$\pi$.  Explicitly for prefix $p = p_1 p_2 \cdots p_k$, we have that $\pi$ is
$p$-acceptable if $\pf{k} = p$ and $\pi_k = N$.  

Then for each prefix $p \in \bigcup_{i=1}^N \SN_i$, define the {\bf strike
payoff} function
\[ S(p) = \sum\limits_{p\text{-acceptable } \pi \in \SN_N}  R_{J(\pi)}(x). \]
As we will see in the proof of Theorem~\ref{t:rnap}, this represents the
(unnormalized conditional) probability of winning the game if the player
follows a strategy that opts to always stop and accept the current candidate
when prefix $p$ is presented.  

To simplify the function, we first observe that the size and $J$-statistic of
the prefix are the only information needed in order to compute the strike
payoffs.

\begin{definition}
Suppose $p$ and $q$ are two prefixes of the same size $k$, and $r$ is a
permutation that has $p$ as a prefix.  Define $\varphi_{p, q}(r)$ to be the
permutation obtained from $r$ by rearranging the first $k$ values of $r$ to
have the relative order given by $q$ instead of the relative order given by
$p$.  
\end{definition}

\begin{example}
Consider prefixes $p=1234$,  $q=2413$, and $r=124635$. Then $\varphi_{p,q}(r)=261435$.
\end{example}

\begin{theorem}\label{t:tos}
If $p$ and $q$ are prefixes with the same size and $J$-statistic, then 
\[ J(\varphi_{p, q}(r)) = J(r) \]
for all permutations $r$ having $p$ as a prefix.
\end{theorem}
\begin{proof}
Suppose that $p$ has size $k$.  Let $r$ be any longer permutation having $p$ as
a prefix, so $\varphi_{p,q}(r)$ is a permutation with the same size as $r$ but
having $q$ as a prefix.

By definition, $J(r)$ is the first bottom-to-top maximum after position $I$ in
$r$, so $r$ must have $J(r)-1$ and all smaller values lying to the left of
position $I$.  Observe that the only ways for $\varphi_{p,q}$ to change the
$J$-statistic are to move one of these smaller values to the right of position
$I$ in $\varphi_{p,q}(r)$, or else to move $J(r)$ to the left of position $I$.

For either of these to occur, we must have $I < k$.  However, the values $1, 2,
\ldots, J(r)-1$ all lie to the left of $k$ so are represented by the same
values in the prefix $p$ as in the prefix $r$.  That is, the flattening
applied to these values in $r$ acts as the identity map.  When we apply
$\varphi_{p,q}$ to $r$, we rearrange all of the values in the first $k$
positions to obtain the relative ordering given by $q$.  So if any of the
values $1, 2, \ldots, J(r)-1$ or $J(r)$ moved past position $I$ it would
contradict our assumption that $J(p) = J(q)$.

Thus, $J(\varphi(r)) = J(r)$ in all cases.
\end{proof}

\begin{definition}
Let $\sigma_J(k, j)$ be the number of $\pi \in \SN_N$ with $J(\pi) = J$ that are $p$-acceptable for some prefix $p$ whose size is $k$ and whose $J$-statistic is $j$.
\end{definition}

Since $\varphi_{p,q}$ restricts to a bijection of
\[ \{p\text{-acceptable}\ \pi \in \SN_N\} \rightarrow \{q\text{-acceptable}\ \pi \in \SN_N\} \] 
that preserves the $J$-statistic by Theorem~\ref{t:tos}, this is well-defined.
Thus, we may rewrite the strike payoff function as
\[ S(p) = \sum_{J=1}^{I+1} \sigma_J(k, j) R_J(x) \]
where $k$ is the size of $p$ and $j$ is $J(p)$.  This is essentially a dot
product of two vectors along ``coordinates'' that are given by values of the
$J$-statistic.  Interestingly, the $\sigma_J(k, j)$ vectors do not depend on $x$
so represent strategic information that applies to any single-step model,
regardless of how severe the reselection step is.  In order to describe these
$\sigma_J(k, j)$ functions explicitly, we first determine the distribution of the
$J$-statistic over all of $\SN_N$.

\begin{lemma}\label{l:mj}
Fix $N$ and $0 \leq I < N$.  The total number of $\pi \in \SN_N$ with $J(\pi) = J$ is given by 
\[ M_J(N, I) := \frac{(N-I) I! (N-J)!}{(I-J+1)!}. \]
\end{lemma}
\begin{proof}
Since $J(\pi) = J$, the value $J$ occurs in some position larger than $I$ and
the values $1, 2, \ldots, J-1$ must occur before the position of $J$.  

We claim that each of the values $1, 2, \ldots, J-1$ must actually occur in
positions that are each less than or equal to $I$.  To see this, suppose for
the sake of contradiction that $1 \leq L \leq J-1$ is the first value occurring
after position $I$ (and before position $J$).  Then each of the values $1, 2,
\ldots, L-1$ occur to the left of $I$ so $L$ would be a bottom-to-top maximum
contradicting that $J(\pi) = J$.  

Thus, the values $1, 2, \ldots, J-1$ can occur anywhere among the positions
less than or equal to $I$ and $J$ can occur anywhere among the final $N-I$
positions.  There are $\binom{N-J}{I-(J-1)}$ ways to choose the remaining
$I-(J-1)$ values for the first $I$ positions and $I!$ ways to order all of
them.  There are $(N-I)!$ ways to order the remaining $N-I$ terms after the
$I$th position, yielding a total count of $\binom{N-J}{I-(J-1)} I! (N-I)!$
which is equivalent to the stated result.
\end{proof}

\begin{theorem}\label{t:main_mu}
Suppose $N$ and $0 \leq I < N$ are fixed.  For $J = 1, 2, \ldots, I+1$, we have
\[ \sigma_J(k, j) = \begin{cases}
        \frac{(N-I)(N-J-1)!(I-1)!}{(I-J)!(k-1)!} = \frac{1}{(k-1)!} M_J(N-1, I-1) & \text{ if $1 \leq k \leq I$ } \\
        \ & \ \\
        \frac{(N-k)(N-J-1)!}{(k-J)!} = \frac{1}{(k-1)!} M_J(N-1, k-1) & \text{ if $I+1 \leq k \leq N$ and $J < j$ } \\
        \frac{(N-j)!}{(k-j)!} = \sum_{T = j}^N \frac{1}{(k-1)!} M_{T}(N-1, k-1) & \text{ if $I+1 \leq k \leq N$ and $J = j$ } \\
        0 & \text{ if ($k \leq I$ and $J = I+1$)  } \\
         & \text{ \ \ or ($I+1 \leq k \leq N$ and $J > j$) } \\
%  We need to be a bit careful with the alternate formulas in terms of M_J...  To get the "basis vector" at k = N, need the (N-k) to be zero explicitly, not sure this happens for the M_J term.
%        0 & \text{ if $k = N$ and $J \neq j$ } \\
%        1 & \text{ if $k = N$ and $J = j$. } \\  these are redundent if we use the explicit (non-recursive) formulas
\end{cases} \]
\end{theorem}
\begin{proof}
First, suppose that $I \geq k$.  We decompose the elements of $\pi$ into {\bf blocks} based on their position, as shown in the schematic illustration below.
\begin{figure}[htbp]
\centering
%For I> the length of p 
\begin{tikzpicture}

% Draw main line
\draw (0,0) -- (10,0);

% Left endpoint
\draw (0,0.2) -- (0,-0.2);
\node[below] at (0,0) {};

% Mark at (3,0)
\draw (3,0.2) -- (3,-0.2);
\node[below] at (3,0) {};

% Mark at (4,0)
\draw (4,0.2) -- (4,-0.2);
\node[below] at (4,0) {};

% Very thick highlighted line between (3,0) and (4,0)
\draw[ultra thick, blue] (3,0) -- (4,0);

% N label slightly lower between (3,0) and (4,0)
\node[above=1pt] at (3.5,0) {$N$};

% I point
\draw (6,0.2) -- (6,-0.2);
\node[below] at (6,0) {$I$};

% Mark at (7,0)
\draw (7,0.2) -- (7,-0.2);
\node[below] at (7,0) {};

% Mark at (8,0)
\draw (8,0.2) -- (8,-0.2);
\node[below] at (8,0) {};

% Very thick highlighted line between (7,0) and (8,0)
\draw[ultra thick, blue] (7,0) -- (8,0);

% j label slightly lower between (7,0) and (8,0)
\node[above=1pt] at (7.5,0) {$J$};

% Right endpoint
\draw (10,0.2) -- (10,-0.2);
\node[below] at (10,0) {$N$};

% Braces and labels
\draw [decorate,decoration={brace,amplitude=5pt,mirror}] (0,-0.5) -- (4,-0.5);
\node[below=7pt] at (2,-0.5) {$k$ block};

\draw [decorate,decoration={brace,amplitude=5pt,mirror}] (4,-0.5) -- (6,-0.5);
\node[below=7pt] at (5,-0.5) {$k^c$ block};

\draw [decorate,decoration={brace,amplitude=5pt}] (0,0.5) -- (6,0.5);
\node[above=7pt] at (3,0.5) {$I$ block};

\draw [decorate,decoration={brace,amplitude=5pt}] (6,0.5) -- (10,0.5);
\node[above=7pt] at (8,0.5) {$N-I$ block};

\end{tikzpicture}
\caption{$I$ is greater than or equal to $k$}
\label{fig:Case1}
\end{figure}
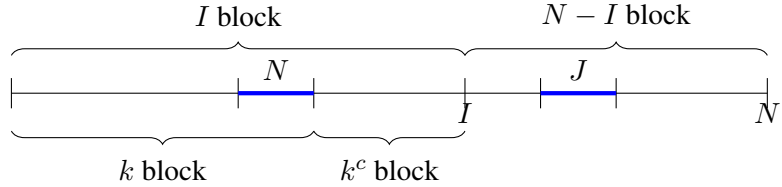

By the same logic as in the proof of Lemma~\ref{l:mj}, the values $1, 2, \ldots, J-1$ must occur in the $I$-block because $J = J(\pi)$.  Since $\pi$ is $p$-acceptable, $N$ also appears in $I$ block.  Therefore, we need to fill $I-J$ remaining positions with values from $N-(J+1)$ available choices, which can be done in $\binom{N-(J+1)}{I-J}$ ways.

To fill out the $k$ block, there are $\binom{I-1}{k-1}$ possible values from the $I$ available choices, because the position of $N$ is already fixed due to being $p$-acceptable. There is only one way to assign the relative ordering of the values in the $k$ block, as determined by the prefix $p$. To fill out the $k^c$ block, the values are already fixed by our prior choices and there are $(I-k)!$ possible ways to order them. Since the $J$ statistic can be placed in any of the $N-I$ positions and there are $(N-I-1)!$ possible ways to order the values we already selected for the $N-I$ block, the total number of $\pi$ for this case is
\[ \binom{N-(J+1)}{I-J} \binom{I-1}{k-1} (I-k)! (N-I) (N-I-1)! \]
which is equivalent to the stated formula.

Next, suppose that $I < k < N$.  Then, the block decomposition for $\pi$ is
shown the schematic illustration below.  In this case we must be careful about
distinguishing between $j = J(p)$ and $J = J(\pi)$.  

\begin{figure}[htbp]
    %For I< the length of p 
    \centering
\begin{tikzpicture}

% Draw main line
\draw (0,0) -- (10,0);

% Left endpoint
\draw (0,0.2) -- (0,-0.2);
\node[below] at (0,0) {};

% I point
\draw (3,0.2) -- (3,-0.2);
\node[below] at (3,0) {$I$};

% Mark at (5,0)
\draw (5,0.2) -- (5,-0.2);
\node[below] at (5,0) {};

% Mark at (6,0)
\draw (6,0.2) -- (6,-0.2);
\node[below] at (6,0) {};

% Very thick highlighted line between (5,0) and (6,0)
\draw[ultra thick, blue] (5,0) -- (6,0);

% N label slightly lower between (5,0) and (6,0)
\node[above=1pt] at (5.5,0) {$N$};

% m point
\draw (6,0.2) -- (6,-0.2);
\node[below] at (6,0) {$k$};

% Right endpoint
\draw (10,0.2) -- (10,-0.2);
\node[below] at (10,0) {$N$};

% Braces and labels
\draw [decorate,decoration={brace,amplitude=5pt}] (0,0.5) -- (6,0.5);
\node[above=7pt] at (3,0.5) {$k$ block};

\draw [decorate,decoration={brace,amplitude=5pt}] (6,0.5) -- (10,0.5);
\node[above=7pt] at (8,0.5) {$k^c$ block};

\end{tikzpicture}
\caption{$I$ is less than $k$}
\label{fig:Case3}
\end{figure}
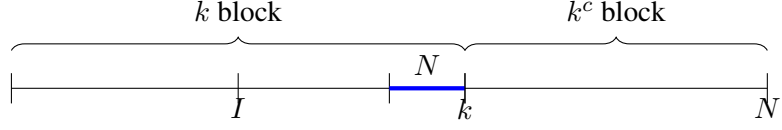

First, note that if $J = J(\pi)$ lies in the $k$ block then all of the values
$1, 2, \ldots, J$ lie within the $k$ block.  When we consider the flattening of
these values for the prefix $p$, we observe that they are each represented by
the same value that they have in $\pi$.  Therefore, we find that $j = J(p) =
J(\pi) = J$ in this case.  Also, when $J = J(\pi)$ lies in the $k^c$ block we
must have $J(\pi) \leq J(p)$ since the values $1, 2, \ldots, J-1$ all lie to
the left of $I$ and so represent the same values in $p$ as they do in $\pi$.

Thus, if $j > J$ then we have that $J$ lies in the $k^c$ block and
we proceed to count these.  Since values $1, 2, \ldots, J-1$ and $N$ already
lie in the $k$ block, there are $\binom{N-(J+1)}{k-J}$ possible ways to select
the remaining values for the $k$ block.  Then, the $J$-statistic can be placed
into any of the $N-k$ positions and there are $(N-k-1)!$ possible ways to order
the values selected for the $k^c$ block.  Thus, there are
$\binom{N-(J+1)}{k-J}\times (N-k) \times (N-k-1)!$ such placements all
together, which agrees with the second formula in the statement.

Moreover, if $j = J$ then we have that the value $J$ can lie in any of the
positions (beyond $I$, except for the position of $N$) in either block of
$\pi$.  So, to find the number of $\pi$ with $j = J$, we must add to the previous
formula the number of permutations $\pi$ having $J$ in the $k$ block.  To fill
out the $k$ block, we need to choose $k-(J+1)$ elements from $N-(J+1)$
available elements since $1, 2, \ldots, J$ and $N$ are all already in the $k$
block.  Then there are $\binom{N-(J+1)}{k-(J+1)}$ possible combinations and
there is only one way to assign the relative order because of the prefix $p$.
Also, there are $(N-k)!$ possible ways to fill out the $k^c$ block.  Thus, we
obtain a total of 
\[ \binom{N-(J+1)}{k-J} (N-k) (N-k-1)! + \binom{N-(J+1)}{k-(J+1)} (N-k)! = \binom{N-J}{k-J} (N-k)!\]
permutations in this case, which agrees with the third formula in the
statement.
\end{proof}

%%%%%%%%%%%%%%%%%%%%%%%%%%%%%%%%%%%%%%%%%%%%%%%%%%%%%%%%%%%%%%%%%%%%%
%  Section
%%%%%%%%%%%%%%%%%%%%%%%%%%%%%%%%%%%%%%%%%%%%%%%%%%%%%%%%%%%%%%%%%%%%%
\bigskip
\section{Deriving the optimal strategy}\label{s:4}

%%%%%%%%%%%%%%%%%%%%%%%%%%%%%%%%%%%%%%%%%%%%%%%%%%%%%%%%%%%%%%%%%%%%%
%  Subsection
%%%%%%%%%%%%%%%%%%%%%%%%%%%%%%%%%%%%%%%%%%%%%%%%%%%%%%%%%%%%%%%%%%%%%
\bigskip
\subsection{The optimal strategy}\label{s:strat}

Next, we want to define an interior payoff that is comparable to the strike payoff $S(p)$
but measures the outcome of {\em rejecting} the prefix flattening $p$, {\em
assuming that the player uses an optimal strategy} on all subsequent moves.
Since we don't yet know the optimal strategy explicitly, we must define this
interior payoff function, which we denote $S^\circ(p)$, recursively as follows.

\begin{definition}\label{d:sint}
Suppose that $p$ is a prefix of size $k$.  Define an {\bf immediate child} of
$p$ to be a prefix $q$ of size $k+1$ that has its first $k$ entries in the same
relative order as $p$.  Each prefix $p$ has $k+1$ immediate children
distinguished by the value in their last position (as the complementary values
must all appear in the relative order specified by $p$).  Only one of the
immediate children ends in a left-to-right maximum (namely, the value $k+1$
itself); denote this child by $\bar p$.  

Then we define the {\bf interior payoff} function
\[ S^\circ(p) = \max\left( S(\bar p), S^\circ(\bar p) \right) + \sum\limits_{\text{other immediate children $q$ of $p$}} S^\circ(q) \]
recursively in terms of $S$ and $S^\circ$ values of prefixes having strictly greater size.
For initial conditions, we use that $S^\circ(p) = 0$ whenever $p$ has size $N$.  (We use the ``topological interior'' notation here to emphasize that the
interior payoff is a sum of certain strike payoffs from the interior of the
``prefix tree'' rooted at $p$, but never includes $S(p)$ itself.)
\end{definition}

We are now in position to compare the two moves available to the player presented with prefix flattening $p$.
\begin{definition}\label{d:posneg}
If $S(p) \geq S^\circ(p)$ we call $p$ a {\bf positive prefix}, otherwise we say
that $p$ is a {\bf negative prefix}.
\end{definition}

\begin{theorem}\label{t:rnap}
The strategy that rejects each negative prefix and accepts the first positive
prefix encountered in the game of best choice is optimal.
\end{theorem}
\begin{proof}
Consider the tree of possible moves for a game of best choice.  Each node
consists of a prefix that we may as well assume ends in a left-to-right maximum
(for otherwise, the strike payoff is zero so the prefix is negative) or else
consists of a complete permutation of size $N$.  These nodes represent the
possible prefix flattenings that are presented to the player throughout the
game.

The nodes are partially ordered by containment of prefixes.  For example, the
root node of the game is the prefix $1$.  This node is directly connected to
prefixes $12$, $213$, $3124$, and others; the prefix $213$ is directly
connected to $2134$ as well as $31425$ and others.  At each node, the player
may accept the current candidate or reject the current candidate.  If no
prefixes are accepted, the player eventually encounters the permutations of
size $N$ and must accept the last candidate (whether they are best or not).

To determine the optimal strategy, begin at the permutations of size $N$.  We
view these as prefixes $p$ that are encountered by a player who has not yet
accepted any candidate.  Then, $S(p)$ is clearly $0$ or $1$ (depending on
whether the current candidate is $N$) while $S^\circ(p)$ is $0$ because there
is no way to reject a prefix that is already size $N$.  Thus, these prefixes
are all positive and accepting any of them is optimal (in fact, it is the
only allowable move left in the game).

Next, we argue by induction.  Suppose that our ``reject negativity, accept
positivity'' strategy is optimal when applied to any prefixes of size $k+1,
k+2, \ldots, N$.  Consider a prefix of size $k < N$.  Such a prefix necessarily
ends in a left-to-right maximum.

Let $\pi \in \SN_N$ represent the true sequence of candidate ranks in the
particular game being presented to the player and consider the moves available
in response to the prefix flattening $p$.  Our probabilities are conditioned on
$\pi \in \SN_N$ having $p$ as a prefix flattening.  By definition, accepting
the current candidate has a probability of winning given by $S(p)$ (divided by
a sum of $R_{J(\pi)}(x)$ terms over all $\pi$ that have $p$ as a prefix
flattening).  On the other hand, by Definition~\ref{d:sint}, rejecting the
current candidate has a probability of winning given by $S^\circ(p)$ (divided
by the same sum of $R_{J(\pi)}(x)$ terms over all $\pi$ that have $p$ as a
prefix flattening).  That is, the formula in Definition~\ref{d:sint} describes
the optimal win probabilities on each subtree from the immediate children; as
these subtrees are disjoint events whose union is the event we are calculating,
we simply sum the results.  We note that determining $S^\circ(p)$ explicitly
does require knowing the optimal strategy from this point in the game onward,
but we may assume that we have already computed $S^\circ(q)$ values and
assigned ``positive'' and ``negative'' status for each prefix $q$ with length
greater than $k$, by the induction hypothesis.  Thus, $S^\circ(p)$ is
well-defined.

Hence, the same ``reject negativity, accept positivity'' strategy is seen to be
optimal for the prefix $p$.  By induction, this strategy is optimal for all the
nodes in the tree.  In particular, $S^\circ(\emptyset) := \max(S(1),
S^\circ(1))$ is the probability of winning the game overall, playing from the
initial prefix flattening $1$ under the optimal strategy.
\end{proof}

Although this setup can be used to solve particular games, we believe that the
$\max$ statement in the recursion for $S^\circ(p)$ will make it unlikely to
find simple expressions or general patterns.  For example, while we could
refine $S^\circ(p)$ by $J$-statistic, defining $\sigma^\circ_J(p)$ via
$S^\circ(p) = \sum_{J=1}^{I+1} \sigma^\circ_J(p) R_J(x)$,
the coefficients $\sigma^\circ_J$ will not share the same kind of
recurrence as in Definition~\ref{d:sint} because the $\max$ statement is not
linear.  Therefore, we consider the following variation.

\begin{definition}\label{d:tc}
Let $T^\circ$ be defined by a recurrence that replaces the $\max$ statement
with the fixed choice $S(\bar p)$:
\begin{equation}\label{e:tc}
 T^\circ(p) = S(\bar p) + \sum\limits_{\text{other immediate children $q$ of $p$}} T^\circ(q). 
\end{equation}
The initial conditions are $T^\circ(p) = 0$ for all $p$ of size $N$.
We refer to this as the {\bf ersatz interior payoff function}.
This function represents the payoff that follows from always rejecting the
current candidate when prefix $p$ is presented and {\em accepting any
subsequent prefix that ends in a left-to-right maximum (i.e. that always
accepts the next ``best so far'' candidate)}.  
\end{definition}

\begin{definition}\label{d:reject}
For a prefix $p \in \bigcup_{i=1}^N \SN_i$, say that $\pi \in \SN_N$ is {\bf
$p$-rejectable} if one of its prefix flattenings is $p$ and rejecting the
prefix flattening $p$ would win the game with interview order $\pi$, assuming
that the player always accepts the next left-to-right maximum occurring after
$p$.  Explicitly, $\pi$ must have a form in which $\pf{k} = p$ and there are no
left-to-right maxima among positions between $\pi_k$ and the position of the
value $N$.  (The value $N$ is always the final left-to-right maximum in $\pi$.)
\end{definition}

\begin{lemma}\label{l:nop2}
Suppose $p$ and $q$ are prefixes with the same size, and the same
$J$-statistic.  Then, $S^\circ(p) = S^\circ(q)$, and $T^\circ(p) = T^\circ(q)$.
\end{lemma}
\begin{proof}
Notice that repeatedly applying the recursive formulas allows us to eventually
write $S^\circ(p)$ and $T^\circ(p)$ as a sum of $S(r)$ values over a set of
prefixes $r$ that all contain $p$ as a prefix.  As $\varphi_{p,q}$ is a
bijection that preserves the $J$-statistic by Theorem~\ref{t:tos}, applying
$\varphi_{p,q}$ to each of these $r$ produces the corresponding sum for
$S^\circ(q)$ and $T^\circ(q)$.  
\end{proof}

\begin{definition}
Let $\tau^\circ_J(k, j)$ be the number of $\pi \in \SN_N$ with $J(\pi) = J$ that are $p$-rejectable for some prefix $p$ whose size is $k$ and whose $J$-statistic is $j$.
\end{definition}

This is well-defined by Lemma~\ref{l:nop2}.  Thus, we may rewrite the interior payoff function as
\[ T^\circ(p) = \sum\limits_{p\text{-rejectable } \pi \in \SN_N}  R_{J(\pi)}(x) = \sum_{J=1}^{I+1} \tau^\circ_J(k, j) R_J(x) \]
where $k$ is the size of $p$ and $j$ is $J(p)$.  Observe that because the
recurrence~(\ref{e:tc}) is linear, it applies to the coefficients
$\tau^\circ_J$ as well.  

\begin{lemma}\label{l:basictau}
We have that $\tau^\circ_J(k,j)$ satisfies the recurrence
\[ \tau^\circ_J(k, j) =  B(k,j) + (k-j+1) \tau^\circ_J(k+1, j) + \sum_{i=1}^{j-1} \tau^\circ_J(k+1, i) \]
for all $k < N$, with the initial condition that $\tau^\circ_J(k,j)$ is the zero vector when $k = N$, and adopting the convention that $j = k$ when $k \leq I$ (recalling that the $J$-statistic of the prefix is technically undefined in this case).
Here, the $B(k,j)$ is an ``affine'' term with two cases:  
\[ B(k, j) = \begin{cases} \sigma_J(k+1, j) & \text{ for all $k \neq I$, } \\ \sigma_J(k+1, k+1) & \text{ if $k = I$. } \end{cases} \]
\end{lemma}
\begin{proof}
We consider a refinement of Equation~(\ref{e:tc}) in which we track how the $J$-statistic changes. 
Fix $N$ and $I$.  Then, we claim that
    \[ \tau^\circ_J(k, j) = \begin{cases}
            \sigma_J(k+1) + k \ \tau^\circ_J(k+1) & \text{ if $k < I$ } \\
            \sigma_J(I+1, I+1) + \sum_{i=1}^{k} \tau^\circ_J(I+1, i) & \text{ if $k = I$ } \\
            \sigma_J(k+1, j) + (k-j+1) \tau^\circ_J(k+1, j) + \sum_{i=1}^{j-1} \tau^\circ_J(k+1, i) & \text{ if $I < k < N$ } \\
            0 & \text{ if $k = N$ }
    \end{cases} \]
for $J = 1, 2, \ldots, I+1$.

The case where $k = N$ is clear, as there is no way to win the game if we
reject the last candidate.

Assume $N > k > I$ and suppose that $p$ is a prefix of size $k$ and
$J$-statistic $j$.  Recall Definition~\ref{d:sint}.  Since $J(p) = j$, we have
$J(\bar{p}) = j$ as well.  Hence, the $p$-rejectable $\pi$ that are
$\bar{p}$-acceptable contribute $\sigma_J(k+1, j)$ to $\tau^\circ_J(k, j)$.

Now suppose that $q$ is an immediate child of $p$ that is not $\bar{p}$.  These
are distinguished by the value in the last position by Definition~\ref{d:sint}.
For $i = 1, 2, \ldots, k$, let $q_{(i)}$ denote the immediate child of $p$
whose value in the last position of $q$ is $i$.  Since $k > I$, placing the
value $i = 1, 2, \ldots, j-1$ at the end of the prefix will ``reset'' the
$J$-statistic, so the children $q_{(i)}$ have $J(q_{(i)}) = i$ for $1 \leq i
\leq j-1$.  These contribute $\sum_{i=1}^{j-1} \tau^\circ_J(k+1, i)$ to
$\tau^\circ_J(k, j)$.

On the other hand, placing the value $i = j, j+1, \ldots, k$ at the end of the
prefix preserves the $J$-statistic, equal to $j$, so the children $q_{(i)}$
have $J(q_{(i)}) = j$ for $j \leq i \leq k$.  These contribute $(k-(j-1))
\tau^\circ_J(k+1, j)$ to $\tau^\circ_J(k, j)$.  This completes the proof of the
formula when $N > k > I$.

Finally, suppose $k \leq I$ and suppose that $p$ is a prefix of size $k$.  The
$p$-rejectable $\pi$ that are $\bar{p}$-acceptable contribute $\sigma_J(k+1,
k+1)$ to $\tau^\circ_J(k, j)$.

There are still $k$ other immediate children of $p$, and placing the value $i =
1, 2, \ldots, j-1$ at the end of the prefix will ``reset'' the $J$-statistic if
$k = I$, so the children $q_{(i)}$ have $J(q_{(i)}) = i$ for $1 \leq i \leq k$
in this case.  If $k < I$, the $J$-statistic of $p$ (or any of its immediate
children) is not defined, but there is no harm in writing the contribution in
the overly-specific form $\sum_{i=1}^{j-1} \tau^\circ_J(k+1, i)$.  Thus, we
obtain the formula for the case $k \leq I$, completing the proof of our
piecewise claim.

The original ``affine'' term from our piecewise formula is 
\[ B(k, j) = \begin{cases} \sigma_J(k+1, j) & \text{if $I < k$} \\ \sigma_J(k+1, k+1) & \text{ if $k \leq I$ } \end{cases} \]
but in fact, $B(k, j)$ can be taken to be $\sigma_J(k+1, j)$ for {\em all} $k$
except the single case when $k = I$ (since $\sigma_J$ doesn't depend on $j$
when $k < I$), in which case we need $A(I, j) = \sigma_J(I+1, I+1)$.
\end{proof}

%%%%%%%%%%%%%%%%%%%%%%%%%%%%%%%%%%%%%%%%%%%%%%%%%%%%%%%%%%%%%%%%%%%%%
%  Subsection
%%%%%%%%%%%%%%%%%%%%%%%%%%%%%%%%%%%%%%%%%%%%%%%%%%%%%%%%%%%%%%%%%%%%%
\bigskip
\subsection{The folding recurrence}\label{s:fold}

In this section, we generalize Theorem~\ref{t:tos} to understand how the $J$-statistic changes under arbitrary permutations of prefix entries.  We will find that all of the $\tau^\circ_J(k, j)$ vectors can already be computed directly from the single vector $\tau^\circ_J(k, I+1)$.

\begin{definition}
Let $q$ and $r$ be any prefixes having the same size $k$ (but not necessarily having the same $J$-statistic).  (Re-)define the function
\[ \f_{q,r}: \{ \text{$q$-prefixed } \pi \in \SN_N \} \rightarrow \{ \text{$r$-prefixed } \pi \in \SN_N\} \] 
by rearranging the first $k$ entries of $\pi$ so that they have the relative order given by $r$ in $\f_{q,r}(\pi)$.
In particular, $\f_{q,r}$ does not change the values nor positions of the last $N-k$ entries of $\pi$ at all.
\end{definition}

It turns out that this generalized $\f_{q,r}$ does not always preserve the $J$-statistic.  In order to track how the $J$-statistic changes under this function, we observe the following.

\begin{lemma}\label{l:pj}
Suppose that $q$ is a prefix of $\pi \in \SN_N$ and let $k$ denote the size of $q$.  If $I < k$ then $J(\pi) \leq J(q)$.  In particular, if $J(\pi) < J(q)$ then the value $J(\pi)$ lies to the right of position $k$ in $\pi$.  (If $I \geq k$ then the prefix imposes no constraints on $J(\pi)$ and we just have $J(\pi) \leq I+1$.)
\end{lemma}
\begin{proof}
Suppose $I < k$.  From the definition of the $J$-statistic, we have that the
values $1, 2, \ldots, J(\pi)-1$ all lie to the left of $I$, hence to the left
of $k$, in $\pi$.  Since these values form an interval of integers from $1$
with no gaps, the roles they play in $\pi$ are the same as the values they have
in any prefix $q$.  That is, the value of each of these entries in $\pi$ is the
same as the value of the corresponding entry in $q$ (that is obtained from the
prefix flattening of $\pi$).

Now suppose that the first claim is false.  Then the observations in the
previous paragraph would apply to the value $J(q)$ since $J(q) < J(\pi)$.  But
this contradicts that $J(q)$ is the first bottom-to-top maximum {\em
after} position $I$ in $q$, by definition.

Next, suppose $J(\pi) < J(q)$.  Then the observations in the first paragraph
imply that each of the values $1, 2, \ldots, J(\pi)-1$ play the same role in
$q$ as they do in $\pi$.  So $J(\pi)$ lies to the right of position $I$ and if
$J(\pi)$ existed among the first $k$ positions of $\pi$, then it would also be
the first bottom-to-top maximum after position $I$ for $q$, being the next
highest value.  But this contradicts that $J(q) \neq J(\pi)$.  Thus, $J(\pi)$
must lie to the right of position $k$ in $\pi$ in this case.
\end{proof}

\begin{theorem}\label{t:sqr}
Suppose $q$ is a prefix for some permutation $\pi$ and $r$ is another prefix of the same size $k$ with $J(q) \geq J(r)$.  Then we have 
    \[
        J(\f_{q,r}(\pi)) = \begin{cases}
            J(\pi) & \text{ if $I \geq k$ or ($I < k$ and $J(\pi) \leq J(r)$) } \\
            J(r) & \text{ if $I < k$ and $J(\pi) > J(r)$ }
        \end{cases}
    \]
\end{theorem}

\begin{remark}
In particular, if $J(q) = J(r)$ then the second case never occurs by Lemma~\ref{l:pj}, so $\f_{q,r}$ always preserves the $J$-statistic in this case (cf. Theorem~\ref{t:tos}).
\end{remark}

\begin{remark}
Equivalently, we can write this formula as $J(\f_{q,r}(\pi)) = \min\left( J(\pi), J(r) \right)$ when $k > I$.
\end{remark}

\begin{proof}
Observe that the only ways for $\f_{q,r}$ to change the $J$-statistic are to move one of the values $1, 2, \ldots, J(\pi)-1$ to the right of position $I$, which lowers the $J$-statistic, or else to move $J(\pi)$ to the left of position $I$, which raises the $J$-statistic.  So if $I \geq k$ then $\f_{q,r}$ can have no effect on the $J$-statistic.  Hence, we may assume $I < k$.

Suppose first that $J(r) < J(\pi)$.  Using the same observations as in the first paragraph of the proof of Lemma~\ref{l:pj}, we find that the values $1, 2, \ldots, J(r)$ all lie to the left of position $k$ in $\pi$ and that each of these entries has the same value in $\f_{q,r}(\pi)$ as it does in $r$.  Hence, $J(\f_{q,r}(\pi)) = J(r)$ in this case.

Next, suppose $J(\pi) \leq J(r)$ and suppose for the sake of contradiction that $\f_{q,r}$ changes the $J$-statistic.  By the first sentence of this proof, there are two ways this can occur.  Using the observations in the first paragraph of the proof of Lemma~\ref{l:pj}, we have that the values $1, 2, \ldots, J(\pi)-1$ all lie to the left of position $I$ in $r$ because $J(r) \geq J(\pi)$.  Hence, none of these values move past position $I$ as we apply $\f_{q,r}$ to $\pi$.  

Therefore, we must have that $J(\pi)$ moves from the right side of position $I$ in $\pi$ to the left side of position $I$ in $\f_{q,r}(\pi)$.  If $J(\pi)$ lies to the right of position $k$, then $\f_{q,r}$ cannot move $J(\pi)$ as $\f_{q,r}$ only affects the entries in the first $k$ positions.  So $J(\pi)$ must lie among the first $k$ positions of $\pi$.  Therefore, the entries $1, 2, \ldots, J(\pi)$ form an interval of integers from $1$ with no gaps.  So each of these entries has the same value in $\pi$ and $\f_{q,r}(\pi)$ as the value of the corresponding entries of $q$ and $r$, respectively, that are obtained from the prefix flattenings.  In particular, $J(q)$ must be equal to $J(\pi)$.

Since moving $J(\pi)$ from the right side of position $I$ to the left side of position $I$ will {\em increase} the $J$-statistic, we violate our assumption that $J(q) \geq J(r)$.  Thus, in each case we obtain a contradiction to the hypothesis that $\f_{q,r}$ changes the $J$-statistic.
\end{proof}

\begin{corollary}\label{c:folding}
Suppose $k > I$ is fixed and $f$ represents $\sigma$ or $\tau^\circ$.  Then we have the ``folding'' recurrence
\[ f_J(k, j) = \begin{cases}
        f_J(k, j+1) & \text{ for all $1 \leq J < j$ } \\
        f_J(k, j+1) + f_{J+1}(k, j+1) & \text{ for $J = j$. } \\
        0 & \text{ for $J > j$. } \\
    \end{cases} \]
\end{corollary}
\begin{proof}
Whenever $J > j$, we have $f_J(k, j) = 0$ by Lemma~\ref{l:pj}.  

Given a prefix $q$ with size $k$ and $J$-statistic $j+1$, we can consider a rearrangement $r$ of $q$ that interchanges the positions of values $j$ and $j+1$.  Then $J(r) = j$.  The function $\f_{q,r}$ does not change acceptability or rejectability of any $\pi \in \SN_N$ as we only interchange two adjacent values that do not involve $N$.  Invoking Theorem~\ref{t:sqr}, we find that the $J$-statistics are mostly unchanged by $\f_{q,r}$ except that all of the $\pi$ having $J$-statistic $j+1$ will ``fold'' into the collection of $\pi$ having $J$-statistic $j$ in the image of $\f_{q,r}$.  Recalling the interpretations of $\sigma_J(k, j)$ and $\tau^\circ(k, j)$ as the number of acceptable (rejectable, respectively) $\pi \in \SN_N$ with $J(\pi) = J$, we obtain the result.
\end{proof}

\begin{remark}\label{r:fold}
Iterating Corollary~\ref{c:folding} allows us to compute all of the $f_J(k, j)$ values directly from $f_J(k, I+1)$.  It also implies that the $f_J(k, j)$ values are completely determined by the ``diagonal'' terms $f_J(k, J)$.  To see this, suppose $J$ is fixed and $j > J$ is arbitrary.  Then we can replace $j$ by $J+1$:  i.e. $f_J(k, j) = f_J(k, J+1)$ by the first part of the Corollary.  Then, the second part of the Corollary implies $f_{J}(k, J) = f_{J}(k, J+1) + f_{J+1}(k, J+1)$, so
\[ f_J(k,j) = \begin{cases} f_J(k,J) - f_{J+1}(k,J+1)  & \text{ if $J < j$, } \\ 0 & \text{ if $J > j$ by Lemma~\ref{l:pj}. } \end{cases} \]
\end{remark}

\begin{example}
To illustrate the results so far, we have computed the $\sigma_J(k, j)$ and $\tau^\circ_J(k, j)$ for $N = 8$ and $I = 3$ in Table~\ref{ta:e}.  By dotting these vectors with 
\[ \scalebox{0.85}{ $R_J(x) = \left[ -\frac{1}{40320}(x - 1)^3, \frac{1}{40320}(7x + 1)(x - 1)^2, -\frac{1}{40320}(21x^2 + 6x + 1)(x - 1), \frac{1}{1152}x^3 + \frac{1}{2688}x^2 + \frac{1}{8064}x + \frac{1}{40320} \right]$ } \]
we can recover the $S(p)$ and $T^\circ(p)$ payoffs for each prefix $p$ having size $k$ and $J$-statistic $j$.
\end{example}

\begin{table}[ht]
\[
\begin{array}{|c c|l|l|}
    \hline
    k & j & \sigma_J(k, j) & \tau^\circ_J(k, j) \\
% N = 8, I = 3
%-- prefix length  1  --
    \hline
1  &  1  &  3600, 1200, 240, 0  &  8820, 3340, 860, 48  \\
%-- prefix length  2  --
    \hline
2  &  2  &  3600, 1200, 240, 0  &  5220, 2140, 620, 48  \\
%-- prefix length  3  --
    \hline
3  &  3  &  1800, 600, 120, 0  &  1710, 770, 250, 24  \\
%-- prefix length  4  --
    \hline
4  &  4  &  480, 240, 96, 24  &  296, 188, 104, 50  \\
4  &  3  &  480, 240, 120, 0  &  296, 188, 154, 0  \\
4  &  2  &  480, 360, 0, 0  &  296, 342, 0, 0  \\
4  &  1  &  840, 0, 0, 0  &  638, 0, 0, 0  \\
%-- prefix length  5  --
    \hline
5  &  4  &  90, 60, 36, 24  &  33, 27, 21, 26  \\
5  &  3  &  90, 60, 60, 0  &  33, 27, 47, 0  \\
5  &  2  &  90, 120, 0, 0  &  33, 74, 0, 0  \\
5  &  1  &  210, 0, 0, 0  &  107, 0, 0, 0  \\
%-- prefix length  6  --
    \hline
\end{array}
\ \ \ \ 
\begin{array}{|c c|l|l|}
    \hline
    k & j & \sigma_J(k, j) & \tau^\circ_J(k, j) \\
    \hline
6  &  4  &  12, 10, 8, 12  &  2, 2, 2, 7  \\
6  &  3  &  12, 10, 20, 0  &  2, 2, 9, 0  \\
6  &  2  &  12, 30, 0, 0  &  2, 11, 0, 0  \\
6  &  1  &  42, 0, 0, 0  &  13, 0, 0, 0  \\
%-- prefix length  7  --
    \hline
7  &  4  &  1, 1, 1, 4  &  0, 0, 0, 1  \\
7  &  3  &  1, 1, 5, 0  &  0, 0, 1, 0  \\
7  &  2  &  1, 6, 0, 0  &  0, 1, 0, 0  \\
7  &  1  &  7, 0, 0, 0  &  1, 0, 0, 0  \\
%-- prefix length  8  --
    \hline
8  &  4  &  0, 0, 0, 1  &  0, 0, 0, 0  \\
8  &  3  &  0, 0, 1, 0  &  0, 0, 0, 0  \\
8  &  2  &  0, 1, 0, 0  &  0, 0, 0, 0  \\
8  &  1  &  1, 0, 0, 0  &  0, 0, 0, 0  \\
    \hline
\end{array}
\]
\caption{The strategic vectors for $N = 8$, $I = 3$}\label{ta:e}
\end{table}

%%%%%%%%%%%%%%%%%%%%%%%%%%%%%%%%%%%%%%%%%%%%%%%%%%%%%%%%%%%%%%%%%%%%%
%  Subsection
%%%%%%%%%%%%%%%%%%%%%%%%%%%%%%%%%%%%%%%%%%%%%%%%%%%%%%%%%%%%%%%%%%%%%
\bigskip
\subsection{Convexity implies reduction from $T^\circ$ to $S^\circ$}\label{s:red}

Finally, we explain how our ersatz $T^\circ$ can be good enough to compute the
true optimal strategy for the game.  Recall the {\bf prefix tree} that was
described in the conceptual proof of Theorem~\ref{t:rnap}.

\begin{definition}
We say that a game of best choice is {\bf strategically convex} if along any
maximal path in the prefix tree, from the root $1$ to a permutation of size
$N$, there is a unique transition from negative prefixes to positive prefixes.
More specifically, if prefixes along the path are denoted $1 = p_{(1)},
p_{(2)}, \ldots, p_{(L)}$, then there exists a unique $k$ such that $p_{(i)}$
are negative for all $i \leq k$ and positive for all $i > k$.  (The $k$ is
allowed to depend on the path.  Also, we have that any prefix $p_{(L)}$ of size
$N$ is always positive by definition.) We refer to the collection of such
prefixes $p_{(k)}$ (over all maximal paths in the prefix tree) as the {\bf
strategic boundary} for the game.
\end{definition}

\begin{definition}\label{d:stratconv}
By replacing $S^\circ$ with $T^\circ$, we define ersatz versions for all of
our strategic terms.  For example, say that a prefix is {\bf
$T^\circ$-positive} if $T^\circ(p) \leq S(p)$ and {\bf $T^\circ$-negative}
otherwise.  Say that a game is {\bf $T^\circ$-convex} if there is a unique
transition from $T^\circ$-negative prefixes to $T^\circ$-positive prefixes
along any path in the prefix tree.
\end{definition}

\begin{theorem}\label{t:reduction}
We have the following statements relating ersatz properties with the real properties of the game:
\begin{enumerate}
    \item $T^\circ(p) \leq S^\circ(p)$ for all prefixes $p$.
    \item If $p$ is $T^\circ$-negative then $p$ is $S^\circ$-negative.
    \item If $p$ is $S^\circ$-negative and every descendant of $p$ is $S^\circ$-positive (so $p$ lies on the strategic boundary), then $T^\circ(q) = S^\circ(q)$ for $q = p$ as well as for all $q$ that are descendants of $p$.
    \item If the game is $T^\circ$-convex then it is strategically convex (with respect to $S^\circ$) as well.
\end{enumerate}
\end{theorem}
\begin{proof}
    To verify the first statement, (reverse) induct on the size of the prefix (from size $N$ down to $1$).  Suppose that $T^\circ(q) \leq S^\circ(q)$ for all the children $q$ of $p$.  Then 
    \[ T^\circ(p) = S(\bar p) + \sum\limits_{\text{other immediate children $q$ of $p$}} T^\circ(q) \]
    \[ \leq \max\left( S(\bar p), S^\circ(\bar p) \right) + \sum\limits_{\text{other immediate children $q$ of $p$}} S^\circ(q) = S^\circ(p) \]
    proving the induction step.

    For the second statement, we have $T^\circ(p) > S(p)$ by hypothesis and $S^\circ(p) \geq T^\circ(p)$ by (1).

    The third statement looks complicated but follows directly from the recursive definitions, in which $\max(S(\bar{p}), S^\circ(\bar{p}))$ is always equal to $S(\bar{p})$ for this part of the tree.  It says that comparing $S$ with ersatz $T^\circ$ will accurately detect the true optimal boundary (that you would have found by comparing $S$ with $S^\circ$).

    For the last statement, suppose that $q$ is a prefix lying on the strategic boundary or lying on the last transition from negative to positive (if there are multiple transitions) and that $p$ is any other prefix of $q$.  By (3), we have that $q$ is $T^\circ$-negative.  By the $T^\circ$-convex hypothesis, we have that $p$ is $T^\circ$-negative.  By (2), we have that $p$ is $S^\circ$-negative.  Since $p$ was arbitrary, this shows that the game is $S^\circ$-convex.
\end{proof}

Thus, it suffices to prove that $T^\circ$ has a unique transition from negative
prefixes to positive prefixes, along any path in the prefix tree.  We will do
this by showing that $T^\circ$ is convex with respect to $j$ (see
Theorem~\ref{t:jconvex}) and $k$ (see Section~\ref{s:kconvex}) separately.

\begin{definition}\label{d:conv}
We say that our game is {\bf $j$-convex} if $T^\circ(p) \leq S(p)$ (that is, $p$ is $T^\circ$-positive) implies $T^\circ(q) \leq S(q)$ (equivalently, $q$ is $T^\circ$-positive) for all prefixes $q$ of the same size $k$ and {\em smaller} $J$-statistic $j$.  We say that our game is {\bf $k$-convex} if $T^\circ(p) > S(p)$ (that is, $p$ is $T^\circ$-negative) implies $T^\circ(q) > S(q)$ (equivalently, $q$ is $T^\circ$-negative) for all prefixes $q$ of the same $J$-statistic $j$ and {\em smaller} size $k$.
\end{definition}

By Lemma~\ref{l:pj}, if $p$ is a prefix of $q$ with size greater than $I$, then
$J(p) \geq J(q)$ and the size of $p$ is obviously smaller than the size of $q$.
(If $p$ has size less than or equal to $I$, the $J$-statistic is undefined.)
Thus, $j$-convexity and $k$-convexity imply $T^\circ$-convexity.  These then
imply by Theorem~\ref{t:reduction}(4) that $S^\circ$ also has a unique
transition from negative prefixes to positive prefixes.  Since $T^\circ$ agrees
with $S^\circ$ from the strategic boundary to the end of the game by
Theorem~\ref{t:reduction}(3), the optimal strategy we obtain from $T^\circ$
will agree with the true optimal strategy (defined via $S^\circ$).  By
Theorem~\ref{t:rnap}, it is therefore optimal to reject $T^\circ$-negative
prefixes and accept the first $T^\circ$-positive prefix that is encountered in
the game.

%%%%%%%%%%%%%%%%%%%%%%%%%%%%%%%%%%%%%%%%%%%%%%%%%%%%%%%%%%%%%%%%%%%%%
%  Section
%%%%%%%%%%%%%%%%%%%%%%%%%%%%%%%%%%%%%%%%%%%%%%%%%%%%%%%%%%%%%%%%%%%%%
\bigskip
\section{Closed formulas and convexity results}\label{s:5}

%%%%%%%%%%%%%%%%%%%%%%%%%%%%%%%%%%%%%%%%%%%%%%%%%%%%%%%%%%%%%%%%%%%%%
%  Subsection
%%%%%%%%%%%%%%%%%%%%%%%%%%%%%%%%%%%%%%%%%%%%%%%%%%%%%%%%%%%%%%%%%%%%%
\bigskip
\subsection{Closed formulas in the ``dot'' basis}\label{ss:dot}

In this section, we introduce closed formulas for the conditional probabilities that govern the optimal strategy.  In particular, these results solve the recurrence given earlier for $\tau^\circ$ in Lemma~\ref{l:basictau}.

Recall from Theorem~\ref{t:1main} that 
\[ R_J(x) = \frac{1}{N!\ (1-x)^{N-I}} \sum_{h=0}^{J-1} {N \choose h} x^h(1-x)^{N-h}. \]
We can view this as a matrix $A$ acting on the basis vectors $x^{J-1}(1-x)^{I-(J-1)}$ (for $J = 1, 2, \ldots, I+1$) where
\[ A_{r, c} = \begin{cases} {N \choose c-1} & \text{ if $c \leq r$ } \\ 0 & \text{ otherwise. } \end{cases} \]
for rows $1 \leq r \leq (I+1)$ and columns $1 \leq c \leq (I+1)$.  Then, we have
\[ N!\ R_J(x) = A_{r,c} \ \dot{R}_J(x) \]
where 
\[ \dot{R}_J(x) := (1-x)^I x^{J-1}(1-x)^{1-J}. \]
Therefore, the dot product $f_J(k, j) \cdot N!\ R_J(x)$ is equal to $\left(A_{r,c}\right)^t \ f_J(k, j) \cdot \dot{R}_J(x)$ by the adjoint property of the matrix transpose, for $f \in \{\sigma, \tau^\circ\}$.  (If $k > I$ and $j < I+1$, the values of $f_J(k,j)$ will be zero for all $J > j$ by Corollary~\ref{c:folding}, so we can restrict $A$ to its leading principal $j \times j$ submatrix in this case.)

We denote these transformed functions $\left(A_{r,c}\right)^t\ f_J(k,j)$ by $\dot{f}_J(k,j)$, and eventually prove the following closed formulas for them.

\begin{theorem}\label{t:rebasis}
    Suppose $k > I$.  Let 
    \[ \dot\sigma_J(k) = {N \choose {J-1}} \frac{(N-J)!}{(k-J)!}, \ \ \text{ and } \ \  \dot\tau_J^\circ(k) = \dot\sigma_J(k) \sum_{i=k-J+1}^{N-J} \frac{1}{i}. \]
    Then, the dot products $\dot\sigma_J(k) \cdot \dot{R}_J(x)$ and $\dot\tau^\circ_J(k) \cdot \dot{R}_J(x)$ yield the same polynomials as $N! (\sigma_J(k) \cdot R_J(x))$ and $N! (\tau^\circ_J(k) \cdot R_J(x))$, respectively.
\end{theorem}

\begin{theorem}\label{t:rebasis_e}
    Suppose $k \leq I$.  Let 
    \[ \dot\sigma_J(k) = \begin{cases} \frac{1}{(k-1)!} {N \choose {J-1}} \frac{(N-J)!(I-1)!}{(I-J)!} & \text{ if $1 \leq J \leq I$, } \\ 0 & \text{ if $J = I+1$, } \end{cases} \]
    \[ \dot\tau^\circ_J(k = I) = \begin{cases} {N \choose {J-1}} \frac{(N-J)!}{(I-J)!} \sum_{i=I-J+1}^{N-J} \frac{1}{i} & \text{ when $1 \leq J \leq I$, } \\ \frac{N!}{I! (N-I)} = {N \choose {J-1}}(N-J)! & \text{ when $J = I+1$, }  \end{cases} \]
    and 
    \[ \dot\tau^\circ_J(k) = \frac{(I-1)!}{(k-1)!} \dot\tau^\circ_J(k=I) + \dot\sigma_J(k) \sum_{i=k}^{I-1} \frac{1}{i} \ \ \text{ when $k < I$. } \]

    Then, the dot products $\dot\sigma_J(k) \cdot \dot{R}_J(x)$ and $\dot\tau^\circ_J(k) \cdot \dot{R}_J(x)$ yield the same polynomials as $N! (\sigma_J(k) \cdot R_J(x))$ and $N! (\tau^\circ_J(k) \cdot R_J(x))$, respectively.
\end{theorem}

\begin{remark}\label{r:dotfold}
These dotted functions $\dot{f} \in \{\dot\sigma, \dot\tau^\circ\}$ are even
more stable under changes in $j$ than the original functions are (cf.
Corollary~\ref{c:folding} and Remark~\ref{r:fold}); their closed formulas do
not include $j$ at all! 

To be precise, when $k > I$ there is technically a dependence on $j$, but only
in the sense that $\dot{f}_J(k,j)$ is $0$ whenever $J > j$ (by
Corollary~\ref{c:folding}).  When $k \leq I$, the $J$-statistic of a prefix is
undefined so the $f_J(k,j)$ do not depend on $j$ at all.  Consequently, we have
suppressed $j$ from our notation in the statement of these results (but will
highlight any intermediate dependence on $j$ in our proofs).
\end{remark}

\begin{remark}\label{r:uddtrans}
Any linear relation or recurrence among the $\sigma$ and $\tau^\circ$ functions (not involving $J$, as $J$ is the ``coordinate variable'' for these vectors) translates to the {\em same} relation among the $\dot\sigma$ and $\dot\tau^\circ$, by applying the linear operator $\left(A_{r,c}\right)^t$, and vice versa.
\end{remark}

\begin{remark}\label{r:agree}
Observe that the two formulas for $\dot\sigma_J$ agree when $k = I$; we prove it as part of Theorem~\ref{t:rebasis_e} only.

The formula for $\dot\tau^\circ_J$ when $k = I$ and $J \leq I$ also satisfies 
\[ \dot\tau_J^\circ(I) = \dot\sigma_J(I) \sum_{i=I-J+1}^{N-J} \frac{1}{i} \]
as in Theorem~\ref{t:rebasis}.  The last entry of this expression is undefined (it includes an indeterminate form $\frac{0}{0}$) but $\dot\tau^\circ_{I+1}(I)$ is $\frac{N!}{I! (N-I)}$ by Theorem~\ref{t:rebasis_e}.
\end{remark}

\begin{remark}\label{r:simple}
When $k \leq I$, we find that the formulas in Theorem~\ref{t:rebasis_e} evidently satisfy
\[ \dot\sigma_J(k-1) = (k-1) \dot\sigma_J(k) \]
and
\[ \dot\tau_J^\circ(k-1) = (k-1) \dot\tau_J^\circ(k) + \dot\sigma_J(k-1) \left(\frac{1}{k-1}\right) = (k-1) \dot\tau_J^\circ(k) + \dot\sigma_J(k). \]
These simple recurrences will be useful for verifying $k$-convexity.
\end{remark}

\begin{proof}{\bf (of the formula for $\dot\sigma_J(k)$ in Theorem~\ref{t:rebasis})}
Suppose that $k > I$.  We can determine 
\[ \dot\sigma_J(k, j) := \left(A_{r,c}\right)^t \ \sigma_J(k, j) \]
using the formulas from Theorem~\ref{t:main_mu}, obtaining
\[ \dot\sigma_J(k, j) = \left(A^t\right)_{J, c} \ \sigma_c(k, j) = \left( \sum_{c = J}^{j-1} {N \choose {J-1}} \frac{(N-k)(N-c-1)!}{(k-c)!} \right) + {N \choose {J-1}} \frac{(N-j)!}{(k-j)!}. \]
The sum inside this expression simplifies as
\[ \sum_{c = J}^{j-1} \frac{(N-c-1)!}{(k-c)!}  
%= (N-k-1)! \sum_{c = J}^{j-1} {{N-c-1} \choose {k-c}} 
= (N-k-1)! \sum_{c = J}^{j-1} {{N-c-1} \choose {N-k-1}} 
= (N-k-1)! \sum_{d = 0}^{j-J-1} {{N-j+d} \choose {N-k-1}}. \]

A standard binomial coefficient summation formula is $\sum_{i=0}^n { {j+i} \choose {k} } = { {n+j+1} \choose {k+1} } - {j \choose {k+1}}$.  Applying this to our situation, we obtain
\[ \sum_{c = J}^{j-1} \frac{(N-c-1)!}{(k-c)!}  
= (N-k-1)!  \left( { { (j-J-1) + (N-j) + 1 } \choose { N-k } } - { {N-j} \choose {N-k} } \right) \]
\[ = (N-k-1)!  \left( { {N-J} \choose {N-k} } - { {N-j} \choose {N-k} } \right). \]
When we substitute this back into our expression for $\dot\sigma_J(k, j)$, we find
\[ \dot\sigma_J(k, j) =  {N \choose {J-1}} (N-k) \left( (N-k-1)!  \left( { {N-J} \choose {N-k} } - { {N-j} \choose {N-k} } \right) \right) + {N \choose {J-1}} \frac{(N-j)!}{(k-j)!} \]
\[ =  {N \choose {J-1}} (N-k)! \left( { {N-J} \choose {N-k} } - { {N-j} \choose {N-k} } \right) + {N \choose {J-1}} \frac{(N-j)!}{(k-j)!}  \]
\[ =  {N \choose {J-1}} \left( \frac{(N-J)!}{(k-J)!} - \frac{(N-j)!}{(k-j)!} \right) + {N \choose {J-1}} \frac{(N-j)!}{(k-j)!}.  \]
This is simply
\[ \dot\sigma_J(k, j) =  {N \choose {J-1}} \frac{(N-J)!}{(k-J)!} = \frac{N!}{(J-1)!(k-J)!(N-J+1)} \]
for each $J = 1, 2, \ldots, j$.
\end{proof}

\begin{proof}{\bf (of the formula for $\dot\sigma_J(k)$ in Theorem~\ref{t:rebasis_e})}
If $k \leq I$, we have 
\[  \sigma_J(k) =  \frac{(N-I)(N-J-1)!(I-1)!}{(I-J)!(k-1)!} = \frac{1}{(k-1)!} M_J(N-1, I-1)  \]
by Theorem~\ref{t:main_mu}.
Then,
\[ \dot\sigma_J(k, j) = \left(A^t\right)_{J, c} \ \sigma_c(k, j) = \sum_{c = J}^{I} {N \choose {J-1}} \frac{(N-I)(N-c-1)!(I-1)!}{(I-c)!(k-1)!} \]
% these are only defined up to I because of the (I-c)!  and they only have I nonzero coords anyway. . .
\[ %=  {N \choose {J-1}} \frac{(N-I)(I-1)!}{(k-1)!} \sum_{c = J}^{I}  \frac{(N-c-1)!}{(I-c)!} 
   =  {N \choose {J-1}} \frac{(N-I)(I-1)!}{(k-1)!} (N-I-1)! \sum_{c = J}^{I} { {N-1-c} \choose {N-I-1} } \]
\[ =  {N \choose {J-1}} \frac{(N-I)!(I-1)!}{(k-1)!} \sum_{d = 0}^{I-J} { {N-I-1+d} \choose {N-I-1} } \]

Using the $\sum_{i=0}^n { {j+i} \choose {k} } = { {n+j+1} \choose {k+1} } - {j \choose {k+1}}$ result again, we get
\[ =  {N \choose {J-1}} \frac{(N-I)!(I-1)!}{(k-1)!} \left( { {I-J + N-I-1 +1} \choose {N-I} } - { {N-I-1} \choose {N-I} } \right) \]
\[ =  {N \choose {J-1}} \frac{(N-I)!(I-1)!}{(k-1)!}  { {N-J} \choose {N-I} }.  \]
%\[ = \frac{N!}{(J-1)!(N-J+1)!} \frac{(N-J)!}{(I-J)!(N-I)!} \frac{(N-I)!(I-1)!}{(k-1)!} \]
%\[ = \frac{N!}{(J-1)!(N-J+1)} \frac{1}{(I-J)!} \frac{(I-1)!}{(k-1)!} \]
This is
\[ \dot\sigma_J(k) = \frac{1}{(k-1)!} {N \choose {J-1}} \frac{(N-J)!(I-1)!}{(I-J)!}. \]
for each $J = 1, 2, \ldots, I$.
\end{proof}

\begin{proof}{\bf (of the formula for $\dot\tau^\circ_J(k)$ in Theorem~\ref{t:rebasis})}
We prove this by induction.  
Since the formulas in Lemma~\ref{l:basictau} are linear, we obtain the same relations among the dotted functions by Remark~\ref{r:uddtrans}.  As the summation in the statement of Theorem~\ref{t:rebasis} will be empty when $k = N$, we find that the formula agrees with Lemma~\ref{l:basictau} and $\dot\tau^\circ_J(N) = 0$, establishing a base case.

Next, suppose that the formula in Theorem~\ref{t:rebasis} is correct for $k+1$ and consider the formula for $\dot\tau^\circ_J(k,j)$ from Lemma~\ref{l:basictau}.  We have
\[ \dot\tau^\circ_J(k,j) = \dot\sigma_J(k+1, j) + (k-j+1) \dot\tau^\circ_J(k+1, j) + \sum_{i=1}^{j-1} \dot\tau^\circ_J(k+1, i) \ \ \text{ if $I < k < N$, } \]
and using the induction hypothesis obtain
\[ = \dot\sigma_J(k+1, j) + (k-j+1) 
    \left( \dot\sigma_J(k+1, j) \sum_{i=(k+1)-J+1}^{N-J} \frac{1}{i} \right)
 + \sum_{h=1}^{j-1} 
 \left( \dot\sigma_J(k+1, h) \sum_{i=(k+1)-J+1}^{N-J} \frac{1}{i} \right).
 \]
By the formula for $\dot\sigma_J$ already proved, we have
\[ \dot\sigma_J(k+1, j) = \frac{1}{k+1-J}\ \dot\sigma_J(k, j) \]
and applying Remark~\ref{r:dotfold} we have
\[ \dot\sigma_J(k+1,h) = \begin{cases} \dot\sigma_J(k+1, j) & \text{ if $h \geq J$ } \\ 0 & \text{ if $h < J$. } \end{cases} \]

Thus, for all $1 \leq J \leq j$ our expression for $\dot\tau^\circ_J(k,j)$ becomes
\[ = \dot\sigma_J(k, j) \left( \frac{1}{k+1-J} + \frac{k-j+1}{k+1-J} \sum_{i=(k+1)-J+1}^{N-J} \frac{1}{i} \right) + \sum_{h=J}^{j-1} \left( \dot\sigma_J(k+1, j) \sum_{i=(k+1)-J+1}^{N-J} \frac{1}{i} \right)  \]
\[ = \dot\sigma_J(k, j) \left( \frac{1}{k+1-J} + \frac{k-j+1}{k+1-J} \sum_{i=(k+1)-J+1}^{N-J} \frac{1}{i} \right) + \frac{j-1-J+1}{k+1-J} \left( \dot\sigma_J(k, j) \sum_{i=(k+1)-J+1}^{N-J} \frac{1}{i} \right)  \]
\[ = \dot\sigma_J(k, j) \left( \frac{1}{k+1-J} + \frac{\left(k-j+1\right) + \left(j-1-J+1\right)}{k+1-J} \sum_{i=(k+1)-J+1}^{N-J} \frac{1}{i} \right) \]
%\[ = \sigma_J(k, j) \left( \frac{1}{k+1-J} + \sum_{i=(k+1)-J+1}^{N-J} \frac{1}{i} \right) = \sigma_J(k, j) \left( \sum_{i=k-J+1}^{N-J} \frac{1}{i} \right) \]
which is $\dot\sigma_J(k, j) \left( \sum_{i=k-J+1}^{N-J} \frac{1}{i} \right)$, as desired.
\end{proof}

\begin{proof}{\bf (of the formulas for $\dot\tau^\circ_J(k)$ in Theorem~\ref{t:rebasis_e})}
By Lemma~\ref{l:basictau} and Remark~\ref{r:uddtrans}, we have
\[ \dot\tau^\circ_J(I) = \dot\sigma_J(I+1, I+1) + \sum_{i=1}^{k} \dot\tau^\circ_J(I+1, i) \ \ \text{ if $k = I$. } \]
By what we've already proved, this is
%\[ = \dot\sigma_J(I+1,I+1) + \sum_{h=1}^{I} \left( \dot\sigma_J(I+1, h) \sum_{i=(I+1)-J+1}^{N-J} \frac{1}{i} \right) \]
\[ = \dot\sigma_J(I+1,I+1) + \sum_{h=J}^{I} \left( \dot\sigma_J(I+1, h) \sum_{i=(I+1)-J+1}^{N-J} \frac{1}{i} \right) \]
\[ = \dot\sigma_J(I+1,I+1) + (I-J+1) \left( \dot\sigma_J(I+1, I+1) \sum_{i=(I+1)-J+1}^{N-J} \frac{1}{i} \right). \]
When $J=I+1$, this is simply $\dot\sigma_{I+1}(I+1,I+1) = \frac{N!}{(N-I) I!}$, while for $1 \leq J \leq I$, we have
%\[ = \dot\sigma_J(I+1,I+1) \left( 1 + (I-J+1) \sum_{i=(I+1)-J+1}^{N-J} \frac{1}{i} \right) \]
\[ = \dot\sigma_J(I+1,I+1)  (I-J+1) \left( \frac{1}{I-J+1} + \sum_{i=(I+1)-J+1}^{N-J} \frac{1}{i} \right) \]
%\[ = \dot\sigma_J(I+1,I+1)  (I-J+1) \left( \sum_{i=I-J+1}^{N-J} \frac{1}{i} \right) \]
\[ = {N \choose {J-1}} \frac{(N-J)!}{(I+1-J)!} (I-J+1) \left( \sum_{i=I-J+1}^{N-J} \frac{1}{i} \right) = {N \choose {J-1}} \frac{(N-J)!}{(I-J)!} \left( \sum_{i=I-J+1}^{N-J} \frac{1}{i} \right) \]
as desired. 

Observe that the formula
\[ \dot\tau^\circ_J(k) = \frac{(I-1)!}{(k-1)!} \dot\tau^\circ_J(k=I) + \dot\sigma_J(k) \sum_{i=k}^{I-1} \frac{1}{i} \]
is true when $k = I$ (since the sum is empty), and that the previous formula provides an explicit form for it.
We argue the general case by induction, so suppose that $k < I$ and that the formula holds for $\dot\tau^\circ_J(k+1)$.
From Lemma~\ref{l:basictau}, we have
\[ \dot\tau^\circ_J(I) = \dot\sigma_J(k+1) + k \ \dot\tau^\circ_J(k+1) \ \ \text{ if $k < I$ } \]
and by the induction hypothesis this is
\[  = \dot\sigma_J(k+1) + k \left( \frac{(I-1)!}{(k)!} \dot\tau^\circ_J(k=I) + \dot\sigma_J(k+1) \sum_{i=k+1}^{I-1} \frac{1}{i} \right) \]
%(using \[ \dot\tau^\circ_J(k) = \frac{(I-1)!}{(k-1)!} \tau^\circ_J(k=I) + \dot\sigma_J(k) \sum_{i=k}^{I-1} \frac{1}{i} \ \ \text{ when $k < I$). } \]
which is
\[  = \dot\sigma_J(k+1) \left(1 + k \sum_{i=k+1}^{I-1} \frac{1}{i}\right) + \frac{(I-1)!}{(k-1)!} \dot\tau^\circ_J(k=I) \]
\[  = \frac{1}{k} \dot\sigma_J(k) \left(1 + k \sum_{i=k+1}^{I-1} \frac{1}{i}\right) + \frac{(I-1)!}{(k-1)!} \dot\tau^\circ_J(k=I). \]
Thus, we obtain $\dot\sigma_J(k) \left(\sum_{i=k}^{I-1} \frac{1}{i}\right) + \frac{(I-1)!}{(k-1)!} \dot\tau^\circ_J(k=I)$ as desired.
\end{proof}

%%%%%%%%%%%%%%%%%%%%%%%%%%%%%%%%%%%%%%%%%%%%%%%%%%%%%%%%%%%%%%%%%%%%%
%  Subsection
%%%%%%%%%%%%%%%%%%%%%%%%%%%%%%%%%%%%%%%%%%%%%%%%%%%%%%%%%%%%%%%%%%%%%
\bigskip
\subsection{Proof of $j$-convexity}\label{s:jconvex}

Here, we prove the first of our $T^\circ$-convexity properties described in Section~\ref{s:red}.

\begin{lemma}\label{l:sp_j}
Suppose $k > I$.  Then, the signs of $\dot\sigma_J(k) - \dot\tau^\circ_J(k)$
are all positive, all negative, or have the form $(+, +, \ldots, +, -, -,
\ldots, -)$ as $J$ varies.  Here, $+$ indicates ``$\geq 0$'' whereas $-$ means
``$< 0$'' in agreement with our ``positive'' and ``negative'' nomenclature for
prefixes.
\end{lemma}
\begin{proof}
By Theorem~\ref{t:rebasis}, we have that $\dot\sigma_J(k) - \dot\tau^\circ_J(k) = \dot\sigma_J \left(1 - \sum_{i=k-J+1}^{N-J} \frac{1}{i}\right)$.
Since $\dot\sigma_J(k)$ is always positive, the signs are controlled by the $\left(1 - \sum_{i=k-J+1}^{N-J} \frac{1}{i}\right)$ term, which is strictly decreasing.
\end{proof}

\begin{theorem}\label{t:jconvex}
Suppose some prefix with size $k$ and $J$-statistic $j$ is $T^\circ$-positive.  Then, each prefix with size $k$ and smaller $J$-statistic is also $T^\circ$-positive.
\end{theorem}
\begin{proof}
In order for the prefix to have a well-defined $J$-statistic, we assume $k > I$ is fixed.  Let $p$ be a prefix with size $k$ and $J$-statistic $j$.  By definition, $T^\circ$-positive means that $S(p) - T^\circ(p) \geq 0$.  This is equivalent to
\[ \sum_{J=1}^j (\dot\sigma_J - \dot\tau^\circ_J) \dot{R}_J \geq 0 \]
using Remark~\ref{r:dotfold} to rewrite the upper limit of the sum and where we have suppressed the argument $(k)$ from each quantity for notational convenience.  Similarly, if $q$ is a prefix with size $k$ and $J$-statistic $(j-1)$, then $S(q) - T^\circ(q) \geq 0$ is equivalent to $\sum_{J=1}^{j-1} (\dot\sigma_J - \dot\tau^\circ_J) \dot{R}_J \geq 0$.  

Therefore, we rewrite our inequality as
\[ \sum_{J=1}^{j-1} (\dot\sigma_J - \dot\tau^\circ_J) \dot{R}_J + (\sigma_j - \tau^\circ_j) \dot{R}_j \geq 0,  \]
and consider the cases for the signs of the last summand.

If $(\dot\sigma_j - \dot\tau^\circ_j) \dot{R}_j$ is negative, then
$\sum_{J=1}^{j-1} (\dot\sigma_J - \dot\tau^\circ_J) \dot{R}_J \geq -(\sigma_j - \tau^\circ_j) \dot{R}_j > 0$
as desired.

Otherwise $(\dot\sigma_j - \dot\tau^\circ_j) \dot{R}_j$ is positive, so $(\dot\sigma_j - \dot\tau^\circ_j)$ is positive, so $(\dot\sigma_J - \dot\tau^\circ_J)$ is positive for each $J < j$ by Lemma~\ref{l:sp_j}.  Thus, 
$\sum_{J=1}^{j-1} (\dot\sigma_J - \dot\tau^\circ_J) \dot{R}_J \geq 0$
in this case as well.
\end{proof}

%%%%%%%%%%%%%%%%%%%%%%%%%%%%%%%%%%%%%%%%%%%%%%%%%%%%%%%%%%%%%%%%%%%%%
%  Subsection
%%%%%%%%%%%%%%%%%%%%%%%%%%%%%%%%%%%%%%%%%%%%%%%%%%%%%%%%%%%%%%%%%%%%%
\bigskip
\subsection{Proof of $k$-convexity}\label{s:kconvex}

Observe that a prefix with length $k$ is $T^\circ$-negative if and only if 
\[ \frac{\dot\tau^\circ_J(k) \cdot \dot{R}_J(x)}{\dot\sigma_J(k) \cdot \dot{R}_J(x)} > 1. \]
Therefore, it suffices to show that
\[ \frac{\dot\tau^\circ_J(k) \cdot \dot{R}_J(x)}{\dot\sigma_J(k) \cdot \dot{R}_J(x)} \geq \frac{\dot\tau^\circ_J(k+1) \cdot \dot{R}_J(x)}{\dot\sigma_J(k+1) \cdot \dot{R}_J(x)} \]
for all $k$ in order to prove the $k$-convexity property of Definition~\ref{d:conv}.

Cross-multiplying, it suffices to show that
\begin{equation}\label{e:k_convex}
 \left(\dot\tau^\circ_J(k) \cdot \dot{R}_J(x)\right) \ \left(\dot\sigma_J(k+1) \cdot \dot{R}_J(x)\right) \ -\ \left(\dot\tau^\circ_J(k+1) \cdot \dot{R}_J(x)\right) \ \left(\dot\sigma_J(k) \cdot \dot{R}_J(x)\right) 
\end{equation}
is nonnegative for all $x \in (0,1)$.  We do this now for the case $k < I$.

\begin{lemma}
Suppose $k < I$.  If a prefix of size $k+1$ is $T^\circ$-negative then prefixes of size $k$ are $T^\circ$-negative.
\end{lemma}
\begin{proof}
Applying 
\[ \dot\sigma_J(k) = k \dot\sigma_J(k+1) \ \ \text{ and } \ \  \dot\tau_J^\circ(k) = k \dot\tau_J^\circ(k+1) + \dot\sigma_J(k+1) \]
from Remark~\ref{r:simple} to Equation~(\ref{e:k_convex}) we obtain
\[ \left( \left( k \dot\tau^\circ_J(k+1) + \dot\sigma_J(k+1) \right) \cdot \dot{R}_J(x)\right) \ \left(\dot\sigma_J(k+1) \cdot \dot{R}_J(x)\right) \ -\ \left(\dot\tau^\circ_J(k+1) \cdot \dot{R}_J(x)\right) \ \left(k \dot\sigma_J(k+1) \cdot \dot{R}_J(x)\right) \]
which is $\left(\dot\sigma_J(k+1) \cdot \dot{R}_J(x)\right)^2$, so is always nonnegative.
\end{proof}

To approach the case $k \geq I$, observe that the coefficient of $x^\ell (1-x)^{2I-\ell}$ in the polynomial (\ref{e:k_convex}) is
\[ \sum_{J1 + J2 = \ell+2} \dot\tau^\circ_{J1}(k) \ \dot\sigma_{J2}(k+1) - \dot\tau^\circ_{J2}(k+1) \ \dot\sigma_{J1}(k) \]
by the usual distributive rule for polynomials.  (We strategically switched parameters in the negative part, but we sum over a set of pairs that is symmetric, so this formulation remains mathematically correct.)  Using the formulas from Theorem~\ref{t:rebasis}, this coefficient is
\begin{equation}\label{e:cp}
    \sum_{\substack{J1 + J2 = \ell+2 \\ 1 \leq J1, J2 \leq j}} \dot\sigma_{J1}(k) \ \dot\sigma_{J2}(k+1) \ \left( \sum_{i=k+1-J1}^{N-J1} \frac{1}{i} - \sum_{i=k+2-J2}^{N-J2} \frac{1}{i} \right) 
\end{equation}
for all $k > I$ and it also holds for $k = I$ as long as $J1 \neq I+1$ by Remark~\ref{r:agree}.

Our goal is to show that each coefficient of $x^\ell (1-x)^{2I-\ell}$, as given
by (\ref{e:cp}), is positive (and it is straightforward to verify that this
holds even when $k = I$ and $J1 = I+1$), but some of the terms in the sum are
empirically negative for some values of $J1$ and $J2$.  Let $\beta(J1,
J2)$ denote the individual term 
\[ \dot\sigma_{J1}(k) \ \dot\sigma_{J2}(k+1) \ \left( \sum_{i=k+1-J1}^{N-J1} \frac{1}{i} - \sum_{i=k+2-J2}^{N-J2} \frac{1}{i} \right) \] 
from this sum.

\begin{lemma}\label{l:b1}
If a term $\beta(J1, J2)$ is negative, we must have $J1 < J2-1$.
\end{lemma}
\begin{proof}
    Since $\dot\sigma$ terms are always positive, the sign of $\beta(J1, J2)$ is entirely determined by the sign of 
    \begin{equation}\label{e:rd}
    \left( \sum_{i=k+1-J1}^{N-J1} \frac{1}{i} - \sum_{i=k+2-J2}^{N-J2} \frac{1}{i} \right).
    \end{equation}

    Consider the contrapositive:  assume $J1 \geq J2-1$ so $-J1 \leq 1-J2$ and the starting points for the summation satisfy $k+1-J1 \leq k+2-J2$.  
    Let $s = J1-J2+1$, representing the difference between these starting points.  Then the uncanceled positive terms from (\ref{e:rd}) are
    \[ \sum_{i=0}^{s-1} \frac{1}{k+1-J1 + i}. \]
    In a similar way, the uncanceled negative terms run from $N-J1+1$ through $N-J2$.  Since $N-J1+1 = (N-J2)-(s-2)$, the uncanceled negative terms from (\ref{e:rd}) are
    % using $-J2 = s-J1-1$, so 
    \[ \sum_{i=0}^{s-2} \frac{-1}{N-J2-i} = \sum_{i=0}^{s-2} \frac{-1}{N-1+s-J1-i} = \sum_{i=0}^{s-2} \frac{-1}{N+1-J1 + \left((s-2)-i\right)}. \]
    Because $N > k$, we can compare the uncanceled positive sum with the uncanceled negative sum term-by-term to find that (\ref{e:rd})
    is positive whenever $J1 \geq J2-1$.
\end{proof}

\begin{definition}
When the parameters for a 1-step reselect game satisfy
\[ I \leq N - \sqrt{N+\frac{5}{4}} + \frac{1}{2}, \]
we say that the game is {\bf simply explicable}.
\end{definition}

Some games are not simply explicable, but the proportion of these tends to $0$ as $N$ becomes large.  (It is an open problem to prove $k$-convexity for the games that are not ``simply explicable.'')  Our goal for the remainder of this section is to prove the following result for simply explicable games, that any negative term from (\ref{e:cp}) can be canceled by a corresponding larger positive term.
\begin{lemma}\label{l:cancel}
Suppose the game is simply explicable.  If $\beta(J1,J2)$ is negative then the absolute value of $\beta(J1,J2)$ is less than or equal to $\beta(J2-1, J1+1)$ (which is necessarily positive).
\end{lemma}

First, note that the ratio of $\dot\sigma_{J2-1}(k) \dot\sigma_{J1+1}(k+1)$ to $\dot\sigma_{J1}(k) \dot\sigma_{J2}(k+1)$ is
\[ \frac{(J1-1)! (J2-1)! (N+1-J1)(N+1-J2)}{(J2-2)!(J1)! (N+2-J2)(N-J1)} = \frac{(J2-1)(N+1-J1)(N+1-J2)}{J1 (N-J1)(N+2-J2)}. \]
Denote this quantity by $\mathcal{R}$.

The sum of the corresponding reciprocal sums is
    \[ \left( \sum_{i=k+2-J2}^{N-J2+1} \frac{1}{i} - \sum_{i=k+1-J1}^{N-J1-1} \frac{1}{i} \right) + \left( \sum_{i=k+1-J1}^{N-J1} \frac{1}{i} - \sum_{i=k+2-J2}^{N-J2} \frac{1}{i} \right) \]
    \[ = \left( \sum_{i=k+2-J2}^{N-J2} \frac{1}{i} + \frac{1}{N-J2+1} - \sum_{i=k+1-J1}^{N-J1} \frac{1}{i} + \frac{1}{N-J1} \right) + \left( \sum_{i=k+1-J1}^{N-J1} \frac{1}{i} - \sum_{i=k+2-J2}^{N-J2} \frac{1}{i} \right) \]
which is simply $\frac{1}{N-J2+1} + \frac{1}{N-J1} = \frac{2N+1-(J1+J2)}{(N-J1)(N+1-J2)}$.  Denote this quantity by $\mathcal{D}$.  
    
Then we obtain
\[ \beta(J2-1,J1+1) = \dot\sigma_{J1}(k) \dot\sigma_{J2}(k+1) \mathcal{R} \left(\mathcal{D} - \left( \sum_{i=k+1-J1}^{N-J1} \frac{1}{i} - \sum_{i=k+2-J2}^{N-J2} \frac{1}{i} \right) \right). \]

Now, assume that $\beta(J1,J2)$ is negative.  Then $J1 < J2-1$ by Lemma~\ref{l:b1} and, contrapositively, $\beta(J2-1, J1+1)$ will be positive as $(J2-1) > (J1+1)-1$.  Also, if $1 \leq J1, J2 \leq j$ with $J1 + J2 = \ell+2$ so $J1$ and $J2$ are valid parameters, then since $J1 < J2-1$ we have $1 \leq J2-1, J1+1 \leq j$ and the sum is preserved so $J2-1$ and $J1+1$ are valid parameters as well.  Thus, our goal
\[ \beta(J2-1,J1+1) > -\beta(J1,J2) \]
is equivalent to showing that $\beta(J2-1,J1+1) + \beta(J1,J2)$ 
\[ = \dot\sigma_{J1}(k) \dot\sigma_{J2}(k+1) \left[ \mathcal{R} \left(\mathcal{D} - \left( \sum_{i=k+1-J1}^{N-J1} \frac{1}{i} - \sum_{i=k+2-J2}^{N-J2} \frac{1}{i} \right)\ \right) + \left(\sum_{i=k+1-J1}^{N-J1} \frac{1}{i} - \sum_{i=k+2-J2}^{N-J2} \frac{1}{i}\right) \right]   \]
is positive.

Hence, it suffices to show $\mathcal{R} \mathcal{D} > (\mathcal{R} - 1) \left( \sum_{i=k+1-J1}^{N-J1} \frac{1}{i} - \sum_{i=k+2-J2}^{N-J2} \frac{1}{i} \right)$.  Dividing by $\mathcal{R}$, which is always a positive quantity, this is equivalent to 
\begin{equation}\label{e:tk}
\mathcal{D} > (1 - 1/\mathcal{R}) \left( \sum_{i=k+1-J1}^{N-J1} \frac{1}{i} - \sum_{i=k+2-J2}^{N-J2} \frac{1}{i} \right). 
\end{equation}

We will return to this inequality at the end of this section to conclude the argument.  It follows easily when $\mathcal{R} \geq 1$ (equivalently, $(1 - 1/\mathcal{R}) \geq 0$), but we run into difficulties otherwise.  So, the remainder of this section is devoted to showing that the case $\mathcal{R} < 1$ does not arise whenever the game is simply explicable.

Define
\[ s = J2 - J1 - 1 \]
representing the ``spread'' between the parameters $J1$ and $J2$, which will be nonnegative by our assumption that $\beta(J1,J2)$ is negative and Lemma~\ref{l:b1}.  From our setup in (\ref{e:cp}), we have
\[ \ell = J1 + J2 - 2. \]
Observe that as $J1$ increases by $1$, we decrease $J2$ by the same amount to preserve $\ell$, so $s$ decreases by $2$.  Hence, when $\ell$ is even, the smallest positive value of $s$ is $1$, otherwise the smallest positive value of $s$ is $2$.  Together, these equations are a change of variables from $(J1,J2)$ to $(\ell, s)$ that we sometimes employ.

Simplifying $1-1/\mathcal{R}$, we obtain
    % wolfram alpha:  1 - 1/((J2 - 1)(N+1-J1)(N+1-J2) / ((J1)*(N-J1)(N+2-J2)))
    %\[ \frac{J1^2 - J1 J2 + J1}{(J2-1)(J1-J2)(N+1-J1)} + \frac{J1^2 - J1 J2 + J1}{(J2-1)(J2-J1)(N+1-J2)} - \frac{J1-J2+1}{J2-1} \]
    %check:  wolfram:  factor 1 - 1/((J2 - 1)(N+1-J1)(N+1-J2) / ((J1)*(N-J1)(N+2-J2)))
    %  and wolfram: J1 * J2 - J1 * N - 2*J1-J2*N-J2+(N+1)^2 - ((N+1-J1-J2)*(N+1) + J1*(J2-1))     check this is zero.
\[ 1-1/\mathcal{R} = \frac{ (J2-J1-1)\left[ (N+1-(J1+J2))(N+1) + J1(J2-1) \right]}{(N+1-J1)(N+1-J2)(J2-1)}. \]
Considering each factor, we find that $1-1/\mathcal{R}$ is negative if and only if 
\[ 0 > (N+1-(J1+J2))(N+1) + J1(J2-1) = (N-1-\ell)(N+1) + \frac{1}{4}\left( (\ell+1)^2 - s^2 \right), \]
and we let $\gamma(\ell, s)$ denote this quantity.

\begin{definition}
In order for $\gamma(\ell, s)$ to be negative, we must have that $(N-1-\ell)$
is negative, so we also have 
\[ \ell > N-1 \geq I+1 \geq j. \]
For each fixed $\ell > j$, we find that $s$ is as large as possible when $J2 =
j$ and $J1 = \ell - j + 2$.  

Let $s_{\mathsf{max}}(\ell) = 2j - \ell - 3$ denote this largest possible value
of $s = J2 - J1 - 1$ (for a given value of $\ell > j$).
\end{definition}

\begin{lemma}
We find that extreme values of $\gamma$ follow an arithmetic progression:
\[ \gamma(\ell-1, s_{\mathsf{max}}(\ell-1)) = \gamma(\ell, s_{\mathsf{max}}(\ell)) + \delta \]
where $\delta = N-j+2$ (which does not depend on $\ell$ nor $s$).  

Also, for the other values of $2 \leq s \leq s_{\mathsf{max}}(\ell)$, we have
\[ \gamma(\ell, s-2) = \gamma(\ell, s) + (s-1). \]

Hence for fixed $\ell$, whenever $\gamma(\ell, s_{\mathsf{max}}(\ell))$ is positive, so are all the other $\gamma(\ell, s)$ values.
\end{lemma}
\begin{proof}
    The second claim follows directly from the formula for $\gamma$ above.  The third claim follows directly from the second claim.
    For the first claim, consider $\gamma(\ell-1, s_{\mathsf{max}}(\ell-1)) - \gamma(\ell, s_{\mathsf{max}}(\ell))$.  By our formulas this is,
    \[ 
        (N-\ell)(N+1) + \frac{1}{4}\left( (\ell)^2 - (2j - \ell - 2)^2 \right) - \left((N-1-\ell)(N+1) + \frac{1}{4}\left( (\ell+1)^2 - (2j - \ell - 3)^2 \right)\right)
    \]
    \[ = N+1 + \frac{1}{4}\left( (-2\ell -1) - (4j -2\ell - 5)\right) = N+1 -j + 1 = \delta \]
    as desired.
    \begin{comment}
    Type
    \begin{verbatim}
    (N-1-L)*(N+1)+(1/4)*( (L+1)^2 - (S)^2 ) - 
      (   (N-1-(L))*(N+1)+(1/4)*( ((L)+1)^2 - (S-2)^2 )   )
    \end{verbatim}

    and

    \begin{verbatim}
    (N-1-L)*(N+1)+(1/4)*( (L+1)^2 - (2*J-L-3)^2 ) - 
      (   (N-1-(L-1))*(N+1)+(1/4)*( ((L-1)+1)^2 - (2*J-(L-1)-3)^2 )   )
    \end{verbatim}

    into Wolfram alpha.
    \end{comment}
\end{proof}

\begin{corollary}\label{c:r1}
Suppose the game is simply explicable, so 
\[ I \leq N - \sqrt{N+\frac{5}{4}} + \frac{1}{2}. \]
Then, $\mathcal{R} > 1$.
\end{corollary}
\begin{proof}
To obtain an initial condition for the arithmetic progression, consider $\ell = 2j-3$ so we have $s_{\mathsf{max}} = 0$.  Then, $\gamma(\ell, s_{\mathsf{max}}) =  (N-1-\ell)(N+1) + \frac{1}{4}(\ell+1)^2$, and suppose this value is negative.  Dividing $\gamma(2j-3, 0)$ by $\delta = N-j+2$ then gives the (negative of the) number of times that $\ell$ can be decreased before $\gamma$ is guaranteed to be positive.  We obtain
\[ \frac{ (N-1-(2j - 3))(N+1) + \frac{1}{4}((2j-3)+1)^2 } { N-j + 2}   = -\left( \frac{(j-1)}{\delta} - \delta + 1 \right) \]
by direct computation\footnote{Expanding and collecting terms, we have:
\[ (-8jN-8j+4N^2+12N+8) + (4j^2 - 8j + 4)  = 4j^2 - 8 jN - 16 j + 4N^2 + 12 N + 12, \]
\[ 4(\delta-j )(N+1) + ((2j-3)+1)^2   = -4(j-1) + 4 \delta^2 - 4 \delta,  \]
\[ (\delta-j )(N+1) + \frac{1}{4}((2j-3)+1)^2   = -\left( (j-1) - \delta^2 + \delta \right). \]}.

\begin{comment}
Things to type into Wolfram alpha:
\begin{verbatim}
( (N-1-(2J- 3))(N+1) + (1/4)*((2J-3)+1)^2 ) / (N-J + 2)
\end{verbatim}

\begin{verbatim}
2J-3-(J-1)/(N-J+2) + (N-J+2) - 1
\end{verbatim}

\begin{verbatim}
(J-1)/(J-N-2) + J + N - 2 >= 2J-2 solve for J
\end{verbatim}

Empirically, this is tight.  
\end{comment}

Hence, whenever 
\[ \ell < (2j-3)-\left( \frac{(j-1)}{\delta} - \delta + 1 \right) = \frac{j-1}{j-N-2} + j + N - 2 \]
we have that $\gamma(\ell, s) > 0$ for all $s > 0$.
Since $\ell \leq 2j-2$, we can solve
\[ 2j-2 \leq \frac{j-1}{j-N-2} + j + N - 2 \]
%\[ (2j-2)(j-N-2) \leq j-1 + (j + N - 2)(j-N-2) \]
for $j$ in terms of $N$, using the quadratic formula.  We find that whenever
\[ j \leq N - \sqrt{N+\frac{5}{4}} + \frac{3}{2}, \]
we have $\gamma(\ell, s) > 0$ for all $\ell$ and $s > 0$.  This holds for all
$j \leq I+1$ by our hypothesis, so $1-1/\mathcal{R}$ is positive, and $\mathcal{R} > 1$.
\end{proof}

\begin{proof}{\bf (of Lemma~\ref{l:cancel})}
    Since our game is simply explicable, we have
    $\mathcal{R} > 1$ by Corollary~\ref{c:r1}.  So, we can obtain (\ref{e:tk}) easily since $\mathcal{D}$ is positive yet 
    \[ \left( \sum_{i=k+1-J1}^{N-J1} \frac{1}{i} - \sum_{i=k+2-J2}^{N-J2} \frac{1}{i} \right) \]
    is negative as in the proof of Lemma~\ref{l:b1}.  Multiplying each of these by $(\mathcal{R}-1)$ and adding $\mathcal{D} > 0$ to each side produces the desired result.
\end{proof}

Consequently, in our parameter regime all of the negative terms in Equation~(\ref{e:cp}) end up being canceled, and so each coefficient of $x^\ell (1-x)^{2I-\ell}$ is positive.  This completes the proof of $k$-convexity, fulfilling the requirement of Theorem~\ref{t:reduction}(4).

\begin{remark}
Although $j$-convexity (Theorem~\ref{t:jconvex}) is primarily a property of $\sigma$ and $\tau^\circ$, in the sense that it holds when we replace $R_J$ by any increasing nonnegative vector, this is {\em not} true of $k$-convexity.  For example, when $N = 16$, $I = 6$, and $j = 7$, we have 
\[ \sigma_J(k = 9) - \tau^\circ_J(k = 9) = (7057428, 3569496, 1614984, 613452, 163380, 1848, -59688), \text{ and } \]
\[ \sigma_J(k = 8) - \tau^\circ_J(k = 8) = (47227776, 18671808, 5751936, 844320, -462720, -477504, -300816). \]
Dotting these with $(1, 1, 1, 1, 1, 1, 225)$ produces a positive result for $k = 8$ and a negative result for $k = 9$, which would violate $k$-convexity.  Of course, $(1, 1, 1, 1, 1, 1, 225)$ is not a legitimate $R_J$ vector.  This example shows that incorporating the essential information from the $R_J$ vector, as we have done in constructing our ``dot'' basis, is necessary in order to prove $k$-convexity.
\end{remark}

\begin{remark}
When $\mathcal{R} < 1$, we find that (\ref{e:tk}) continues to hold for many small parameter values.  However, there are examples such as $N = 9116$, $I = 9100$, $j = 9101$, $k = 9102$ and $(J1, J2) = (8982, 9101)$, for which
    \[ \mathcal{D} \approx  0.0699626865671642, \ \  (1 - 1/\mathcal{R}) \approx -0.0409542124542125, \ \text{ and } \]
    \[ \left( \sum_{i=k+1-J1}^{N-J1} \frac{1}{i} - \sum_{i=k+2-J2}^{N-J2} \frac{1}{i} \right) \approx -1.70831511339475, \]
so the inequality (\ref{e:tk}) is not satisfied:  ${\bf 0.06996268}65671642 \not > {\bf 0.06996270}00927106$.

Thus, for games that are not simply explicable, it may be possible that (\ref{e:cp}) is still always positive by some more complicated argument involving canceling the negative beta term by more than one positive term, or by a different positive term than the one that we used when $R > 1$.  Alternatively, it is possible that all games are $k$-convex but that using the ratios $\frac{\dot\tau^\circ_J(k) \cdot \dot{R}_J(x)}{\dot\sigma_J(k) \cdot \dot{R}_J(x)}$ that led to (\ref{e:cp}) is not always a way to discern this.  Finally, it may be the case that (some of) these games are not $k$-convex at all.  It remains an open problem to investigate these games further.
\end{remark}

%%%%%%%%%%%%%%%%%%%%%%%%%%%%%%%%%%%%%%%%%%%%%%%%%%%%%%%%%%%%%%%%%%%%%
%  Section
%%%%%%%%%%%%%%%%%%%%%%%%%%%%%%%%%%%%%%%%%%%%%%%%%%%%%%%%%%%%%%%%%%%%%
\bigskip
\section{Solutions and applications}\label{s:6}

\subsection{Specifics}\label{s:themain}
We can now use the results we've developed to solve for the optimal strategy in our 1-step reselect games.

\begin{definition}
Say that a prefix of length $k$ is {\bf winnable} if it ends in a left-to-right maximum.  That is, the entry with value $k$ must occur in position $k$.
\end{definition}

From the rules of the game, we never win if we accept a prefix that is not winnable so we must restrict our attention to the winnable prefixes.

\begin{definition}
Let $\dot\Delta_J(k,j)$ be $\dot\sigma_J(k) - \dot\tau^\circ_J(k)$ if $J \leq j$, or zero if $J > j$.
\end{definition}

This $\dot\Delta_J$ function contains the essential strategic information about the game:  Whenever the dot product $\dot\Delta_J(k,j) \cdot \dot{R}_J(x) \geq 0$, a prefix with parameters $(k,j)$ is positive (in the strategic sense:  that it is optimal to accept it), whereas $\dot\Delta_J(k,j) \cdot \dot{R}_J(x) < 0$ implies that a prefix with parameters $(k,j)$ is strategically negative.  The convexity results from Section~\ref{s:5} imply that the set of $(k,j)$ that are strategically positive forms a geometrically convex subset of $\{1, \ldots, N\} \times \{1, \ldots, I+1\}$.

We can put our results together as follows to obtain a fast formula for computing strategies and win probabilities in any simply explicable 1-step reselect game.

\begin{theorem}\label{t:themain}
Suppose the parameters of our game satisfy $I \leq N - \sqrt{N+\frac{5}{4}} +
\frac{1}{2}$, and consider each prefix $p$, of length $k$ and $J$-statistic $j$,
say.  When $k > I$, we have that $\dot\Delta_J(k, j) \cdot \dot{R}_J(x)$ is
\[ \left\langle {N \choose {J-1}} \frac{(N-J)!}{(k-J)!}  \left( 1-\sum_{i=k+1-J}^{N-J} \frac{1}{i} \right)\right\rangle_{J=1}^j \cdot \left\langle x^{J-1}(1-x)^{I-(J-1)} \right\rangle_{J=1}^j. \]

When $k \leq I$, we have that $\dot\Delta_J(k) \cdot \dot{R}_J(x)$ is
\[ \left\langle \frac{1}{(k-1)!} {N \choose {J-1}} \frac{(N-J)!(I-1)!}{(I-J)!}  \left( 1-\sum_{i=k}^{I-1} \frac{1}{i}-\sum_{i=I+1-J}^{N-J} \frac{1}{i} \right)\right\rangle_{J=1}^I \cdot \left\langle x^{J-1}(1-x)^{I-(J-1)} \right\rangle_{J=1}^I \]
\[ - \frac{1}{(k-1)!} \frac{N!}{I (N-I)} x^I. \]

The optimal strategy is to reject candidates until we encounter a winnable
prefix for which $\dot\Delta_J(k, j)\cdot\dot{R}_J(x)$ is positive for the
first time; then accept the current candidate (at position $k$ from the prefix
with $J$-statistic $j$).
\end{theorem}
\begin{proof}
    We use Theorems~\ref{t:rebasis} and \ref{t:rebasis_e} to write the explicit formulas and follow the outline given at the end of Section~\ref{s:red}.  By Theorem~\ref{t:jconvex} and the results from Section~\ref{s:kconvex}, we find that our simply explicable 1-step reselect game is $j$-convex and $k$-convex, so it is $T^\circ$-convex.  According to Theorem~\ref{t:reduction}(4), this means it is strategically convex with respect to $S^\circ$ as well.  By Definition~\ref{d:stratconv}, this means there is a unique transition along any path of prefixes through the game tree from being strategically negative (for which it is optimal to reject) to strategically positive (for which it is optimal to accept).  Since $T^\circ$ agrees with $S^\circ$ from the strategic boundary to the end of the game by Theorem~\ref{t:reduction}(3), the strategy we have given based on $\dot\Delta_J(k, j)$ is optimal by Theorem~\ref{t:rnap} (even though we use $T^\circ = \frac{1}{N!} \left(\dot\tau_J^\circ \cdot \dot{R}_J\right)$ in place of $S^\circ$).
\end{proof}

\begin{lemma}\label{l:muwin}
The number of winnable prefixes with parameters $(k,j)$ is
\[
    \mu(k,j) := \begin{cases}
        \frac{(k-I-1) I! (k-j-1)!}{(I-j+1)!} & \text{ if $k > I+1$, } \\
        0 & \text{ if $k = I+1$ and $j \leq I$, } \\
        (k-1)! & \text{ otherwise. } 
    \end{cases}
\]
\end{lemma}
\begin{proof}
If $k \leq I$, the $J$-statistic is undefined, so $j$ plays no role.  The restriction that the prefix ends in a left-to-right maximum fixes one entry, and the other $k-1$ entries may be permuted arbitrarily.  If $k = I+1$, the last entry of the prefix is also the value of the $J$-statistic, so the prefix is winnable in this case only when $j = I+1$, resulting in the stated formulas.  If $k > I+1$, we may apply the formula from Lemma~\ref{l:mj} with $N = k-1$.
\end{proof}

\setcounter{MaxMatrixCols}{20}
\begin{example}\label{e:3010}
    For $N = 30$, $I = 10$, $x = 0.5$, we obtain the following division of prefix parameters into ``strategically positive'' (+) and ``strategically negative'' (-) by recording the signs of $\dot\Delta_J(k, j) \cdot \dot{R}_J(x)$ according to Theorem~\ref{t:themain}:
    \[ \scalebox{0.7}{ $\begin{matrix}
        k \backslash j & 1 & 2 & 3 & 4 & 5 & 6 & 7 & 8 & 9 & 10 & 11 \\
        1 & - & - & - & - & - & - & - & - & - & - & - \\
        2 & - & - & - & - & - & - & - & - & - & - & - \\
        \cdots & . & . & . & . & . & . & . & . & . & . & . \\
        10 & - & - & - & - & - & - & - & - & - & - & - \\
        11 & \emptyset & \emptyset & \emptyset & \emptyset & \emptyset & \emptyset & \emptyset & \emptyset & \emptyset & \emptyset & - \\
        12 & + & + & - & - & - & - & - & - & - & - & - \\
        13 & + & + & + & - & - & - & - & - & - & - & - \\
        14 & + & + & + & + & + & - & - & - & - & - & - \\
        15 & + & + & + & + & + & + & + & - & - & - & - \\
        16 & + & + & + & + & + & + & + & + & + & + & - \\
        17 & + & + & + & + & + & + & + & + & + & + & + \\
        18 & + & + & + & + & + & + & + & + & + & + & + \\
        \cdots & . & . & . & . & . & . & . & . & . & . & . \\
        29 & + & + & + & + & + & + & + & + & + & + & + \\
        \end{matrix}$ } \]

        Observe that this diagram gives a complete description of the optimal strategy for the given parameters.  To understand the $\emptyset$ notation, recall that values from row $k = I+1$ with $j \leq I$ have $\mu(k,j) = 0$ because there will be no winnable prefixes there.  Thus, no strategy will ever accept a prefix with these parameter values, regardless of the sign of $\dot\Delta_J(k,j) \cdot \dot{R}_J(x)$.

We could describe the strategy more concisely by just providing the threshholds for each value of $j$; for example, among prefixes with $j = 3$ we should reject the first $12$ candidates and accept the next best subsequent candidate.  Even more succinctly, we could provide the corner entries (rightmost ``+'' in each row, say) and infer the rest of the diagram using our convexity results.
\end{example}

We next consider how to compute the total probability of winning the game, following this optimal strategy. 

From Theorem~\ref{t:themain}, if any prefix of length $k \leq I$ is
strategically positive, then the optimal strategy does not depend on $j$ so is positional (as in the
classical uniform distribution game): it is always optimal to reject the first $k-1$
candidates and accept the next left-to-right maximum ranked candidate that
appears, where $k$ is the length of the smallest positive prefix.
In our setup, the simplest way to write the win probability for the optimal strategy in this case is 
\begin{equation}\label{e:mainprob_t}
 \frac{(k-1)!}{N!} \ \left( \dot\tau_J(k-1, I+1) \cdot \dot{R}_J(x) \right). 
\end{equation}
(We multiply by the total number of prefixes of size $k-1$ since we reject all of them, not just the winnable ones.  Also, note that $S^\circ$ and $T^\circ$ will diverge prior to length $k-1$, because $T^\circ$ does not have the $\max$ statements, so we cannot merely compute $\dot\tau_J(1, I+1) \cdot \dot{R}_J(x)$ though this quantity does serve as a lower bound by Theorem~\ref{t:reduction}(1).)

On the other hand, if no prefix of length $k \leq I$ is stategically positive, then we must use the following formula.

\begin{theorem}\label{t:themainp}
The probability of winning the game using the optimal strategy is 
\begin{equation}\label{e:mainprob_d}
 \sum_{\text{ positive } (k,j)} \frac{\mu(k,j)}{N!} \ \left( \dot\Delta_J(k, j) \cdot \dot{R}_J(x) \right) = \sum_{\substack{1 \leq k \leq N \\ 1 \leq j \leq I+1}} \frac{\mu(k,j)}{N!} \ \max\left( 0, \dot\Delta_J(k, j) \cdot \dot{R}_J(x) \right).
\end{equation}
\end{theorem}
\begin{proof}
Our goal is to sum the $R_{J(\pi)}(x)$ contributions from the $\pi \in \S_N$
that are won using the optimal strategy, expressed in a ``+/- diagram'' as in
Example~\ref{e:3010}, say.  Recalling our definitions from
Sections~\ref{s:strike} and \ref{s:strat}, observe that each permutation $\pi
\in \S_N$ contributes a positive contribution to 
\[ S(p) = \sum\limits_{p\text{-acceptable } \pi \in \SN_N}  R_{J(\pi)}(x) \]
for a unique prefix $p$:  to construct it, we flatten the entries of
$\pi$ up to and including the position of the value $N$.  Similarly, each
permutation $\pi \in \S_N$ contributes a positive contribution to 
\[ T^\circ(q) = \sum\limits_{q\text{-rejectable } \pi \in \SN_N}  R_{J(\pi)}(x) \]
for a unique prefix $q$ by Definition~\ref{d:reject}.  To construct it,
we flatten the entries of $\pi$ up to and including the position of the
penultimate left-to-right maximum in $\pi$.

The crucial point is that by accepting a prefix $p$, we forgo the opportunity
to accept any subsequent prefix $q$ that contains $p$ as a prefix.  For
example, consider two permutations from $\S_N$ illustrated schematically,
\[ \pi =  [\underbrace{\ \_\ \_\ \_\ {\bf N}\ }_{\text{prefix $p$}} \_\ \_\ \_\ \_\ \_\ \_\ \_\ ],  \]
\[ \pi' =  [\underbrace{\ \underbrace{\ \_\ \_\ \_\ {\bf L}\ }_{\text{prefix $p$}} \underbrace{\times \times \times}_{\text{no LRMs}} {\bf N}\ }_{\text{prefix $q$}} \_\ \_\ \_\ ],  \]
and assume that $p$ is a positive prefix according to our optimal strategy.
Here, $L$ is the penultimate left-to-right maximum in $\pi'$ so $\pi'$ is
$p$-rejectable.  Then we can win $\pi$ (by accepting at prefix $p$) but we will
never be able to capture $\pi'$, even though it is $q$-acceptable (and $q$ must
be positive by convexity), because our strategy will already have accepted
prefix $p$ (and failed to win in this case).

Therefore, we must sum over the $\pi \in \S_N$ that are $p$-acceptable for some
strategically positive prefix $p$ but not $q$-rejectable for any strategically
positive prefix $q$.  Hence, the the probability of winning the game is
\[ \sum\limits_{\text{positive prefixes $p$}} S(p) - \sum\limits_{\text{positive prefixes $q$}} T^\circ(q) = \sum\limits_{\text{positive prefixes $p$}} S(p) - T^\circ(p) \]
which is equal to the stated expressions.
\end{proof}

\begin{example}
For the parameters $N = 30$, $I = 10$, $x = 0.5$ as in Example~\ref{e:3010}, the optimal probability of success is $0.37977053837147$.  In order to compute this, we dot product the $\dot\Delta_J(k, j) \cdot \dot{R}_J(x)$ values corresponding to the positive entries in our ``+/- diagram'' with the $\mu(k,j)$ values from Lemma~\ref{l:muwin} (and scale by $N!$).

Although it is not an issue for this computation, keep in mind that values from row $k = I+1$ with $j \leq I$ have $\mu(k,j) = 0$, so will never contribute to the win probability. %(regardless of their sign in the ``+/- diagram'').
\end{example}

\subsection{Generalities}\label{s:stratbound}

Intuitively, $\dot\Delta_J(k,j) \cdot \dot{R}_J(x)$ must be close to zero for
$(k,j)$ on the strategic boundary between positive and negative prefixes.
In order for this to occur, we must then have different signs among the
coordinates of $\dot\Delta_J(k)$ (for fixed $k$, as $J$ varies), since the
coordinates of $\dot{R}_J(x)$ are always nonnegative for $0 \leq x \leq 1$.  
In this section, we derive bounds on where these sign changes can occur.

To characterize the $k$ for which $\dot\Delta_J(k)$ can have mixed signs (as
$J$ varies), consider two classical specializations that are embedded in our
1-step reselect model.  At $x = 0$, we implement no filtering at all in our
reselect step.  Therefore, the game is the same as the classical (uniform) game
of best choice with $N$ candidates.  We have $\dot{R}_J(0) = (1, 0, \ldots, 0, 0)$
and Theorem~\ref{t:themain} specializes to 
\[ \dot\Delta_J(k) \cdot \dot{R}_J(0) = \frac{(N-1)!}{(k-1)!}  \left( 1-\sum_{i=k}^{N-1} \frac{1}{i} \right) \]
in both cases (for $k > I$ and for $k \leq I$).

As $x$ becomes arbitrarily close to $1$, we approach $\dot{R}_J(1) = (0, 0, \ldots, 0, 1)$.
Essentially, the model filters so strongly that it becomes almost certain that
the best candidate will be found among the last $N-I$ positions.  
Theorem~\ref{t:themain} specializes to 
\[ \dot\Delta_J(k) \cdot \dot{R}_J(1) = {N \choose {I}} \frac{(N-I-1)!}{(k-I-1)!}  \left( 1-\sum_{i=k-I}^{N-I-1} \frac{1}{i} \right) \]
when $k > I$ (and is always negative when $k \leq I$).  Evidently, these do not
depend on $J$ so the optimal strategies are positional (i.e. depend only on
$k$, the number of candidates interviewed so far), just as in the classical
game.

\begin{definition}
Let $k^\ast_{\mathsf{classical}}(N)$ denote the last position that is rejected under the optimal strategy for the classical game with $N$ candidates.
\end{definition}
% N.B. so our k* is equal to GM's s*-1.

From the results already worked out in the 1960's (see Gilbert--Mosteller \cite{gilbert--mosteller} for example), we know that $k^\ast_{\mathsf{classical}}(N)$ is the maximal value of $k$ such that 
\begin{equation}\label{e:x0}
 \sum_{i=k}^{N-1} \frac{1}{i} > 1, 
\end{equation}
and so $1 - \sum_{i=k}^{N-1} \frac{1}{i}$ is negative if and only if $k \leq k^\ast_{\mathsf{classical}}(N)$.
This matches the expression in our $x = 0$ model.  By the same result, we find
that $I + k^\ast_{\mathsf{classical}}(N-I)$ is the maximal value of $k$ such
that 
\begin{equation}\label{e:x1}
 \sum_{i=k-I}^{N-I-1} \frac{1}{i} > 1, 
\end{equation}
which matches the expression in our $x \rightarrow 1$ model.

Next, we extend these observations to all $J$.  We then find that the $k$ for
which $\dot\Delta_J(k)$ can have mixed signs (among the various values of $J$)
are restricted to a relatively narrow band that begins at the classical
transition $k = \frac{N}{e}$ and then extends for another $I\left(1 -
\frac{1}{e}\right)$ positions.  Therefore, these are the only positions that
can be part of the strategic boundary for $\dot\Delta_J(k,j) \cdot
\dot{R}_J(x)$.

\begin{theorem}
If $k \leq k^\ast_{\mathsf{classical}}(N)$ then $\dot\Delta_J(k)$ is negative for all $J$.  If $k > I + k^\ast_{\mathsf{classical}}(N-I)$ then $\dot\Delta_J(k)$ is positive for all $J$.
\end{theorem}
\begin{proof}
For each $J$, the sign of $\dot\Delta_J(k)$ is equal to the sign of $\left( 1-\sum_{i=k+1-J}^{N-J} \frac{1}{i} \right)$ when $k > I$, or $\left( 1-\sum_{i=k}^{I-1} \frac{1}{i}-\sum_{i=I+1-J}^{N-J} \frac{1}{i} \right)$ when $k \leq I$ by Theorem~\ref{t:themain}.

If $k \leq k^\ast_{\mathsf{classical}}(N)$ then we know $1 - \sum_{i=k}^{N-1} \frac{1}{i}$ is negative by the classical results described for (\ref{e:x0}).  If $k > I$, the sign of $\dot\Delta_J(k)$ is equal to the sign of 
\[ \left( 1-\sum_{i=k+1-J}^{N-J} \frac{1}{i} \right) \leq 1 - \sum_{i=k}^{N-1} \frac{1}{i} < 0, \]
while if $k \geq I$, it is equal to the sign of 
\[ \left( 1-\sum_{i=k}^{I-1} \frac{1}{i}-\sum_{i=I+1-J}^{N-J} \frac{1}{i} \right) \leq \left( 1-\sum_{i=k+1-1}^{N-1} \frac{1}{i} \right) < 0, \]
comparing term-by-term for $N-k$ total terms.

If $k > I + k^\ast_{\mathsf{classical}}(N-I)$ then we know $1 - \sum_{i=k-I}^{N-I-1} \frac{1}{i}$ is nonnegative by the classical results described for (\ref{e:x1}), and we must have $k > I$.  Therefore, the sign of $\dot\Delta_J(k)$ is equal to the sign of 
\[
 \left( 1-\sum_{i=k+1-J}^{N-J} \frac{1}{i} \right) \geq \left( 1-\sum_{i=k+1-(I+1)}^{N-(I+1)} \frac{1}{i} \right) \geq 0,
\]
as desired.
\end{proof}

\begin{corollary}\label{c:band}
In the 1-step reselect game, for sufficiently large $N$, we have that the position $k$ of first acceptance under the optimal strategy occurs within the interval $\frac{N}{e} \leq k \leq \frac{N}{e} + I \left(1-\frac{1}{e}\right)$.
\end{corollary}
\begin{proof}
When evaluating $\dot\Delta_J(k,j) \cdot \dot{R}_J(x)$ for any $j$, the first acceptance cannot occur earliar than $k = \frac{N}{e}$ prior to which all the $\dot\Delta_J(k)$ have only negative coordinates, and it can't occur later than $k = \frac{N}{e} + I \left(1 + \frac{1}{e}\right)$ after which the $\dot\Delta_J(k)$ have only positive coordinates.
\end{proof}

\begin{example}
When $N = 30$ and $I = 10$, these bounds are approximately $11 \leq k \leq 17$ which agrees well with Example~\ref{e:3010}.
\end{example}

In fact, we can apply the same classical results we used to derive (\ref{e:x1}) to find the transition for any $J$, in the case where $k > I$.  That is, $J-1 + k^\ast_{\mathsf{classical}}(N-J+1)$ is the maximal value of $k$ such that 
\[ \left( 1-\sum_{i=k+1-J}^{N-J} \frac{1}{i} \right) = \text{ sign of } \dot\Delta_J(k) \]
is negative.  This implies that the boundary between positive and negative values of $\dot\Delta_J(k)$ is asymptotically a straight line with slope $1-\frac{1}{e}$, at least for the case $k > I$.  That is, for $N$ sufficiently large, we have that $\dot\Delta_J(k)$ is positive if and only if 
\[ k \geq \left(1-\frac{1}{e}\right) J + \frac{N}{e} + \frac{1}{e} - 1. \]
Note that this is not the same as the more interesting boundary in our ``+/- diagram'' from Example~\ref{e:3010}, however, because we have not accounted for effects from the dot product with $\dot{R}_J(x)$.

For the case where $k \leq I$, one could apply elementary integral estimates as in Gilbert--Mosteller \cite{gilbert--mosteller} to
$\left( 1-\sum_{i=k}^{I-1} \frac{1}{i}-\sum_{i=I+1-J}^{N-J} \frac{1}{i} \right)$
in order to obtain an asymptotic expression for the boundary between positive and negative values of $\dot\Delta_J(k)$.  We have not done this carefully, but this case does not generally produce a straight line boundary.

\subsection{Examples}

Applying results of Section~\ref{s:themain} for particular choices of $N$ and $I$, we now demonstrate how to find optimal strategies and probabilities for all $x$ simultaneously.  

First, observe that there are only finitely many possible strategies that a player may use.  Each of these are encoded by a ``+/- diagram'' of the form illustrated in Example~\ref{e:3010}.  In order to belong to the class of potentially optimal strategies there are certain constraints these must obey, such as convexity and the restrictions of Corollary~\ref{c:band}.  Nevertheless, it is clear there are only finitely many feasible diagrams/strategies for a given $N$ and $I$.  Then each such strategy gives rise to a polynomial probability of success via the formulas (\ref{e:mainprob_t}) and (\ref{e:mainprob_d}).

If we plot these together on the same axes for $0 \leq x \leq 1$, we can see how the optimal strategy evolves.  At each fixed value of $x$, the strategy corresponding to the highest curve in the collection is optimal for that value of $x$.  Thus, we can look along the top of the collection of curves to find the function that returns the win probability under optimal play.  Whenever two maximal curves cross, we have a transition in the optimal strategy.  This is illustrated for two sets of parameters in Figure~\ref{f:pc}.

\begin{figure}[h]
\includegraphics[scale=0.45]{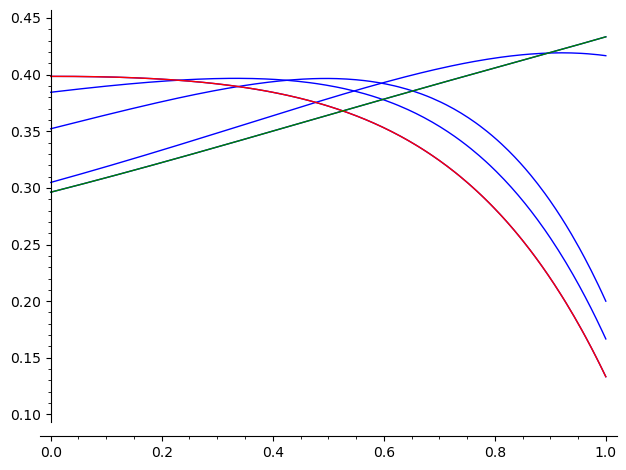}
\includegraphics[scale=0.45]{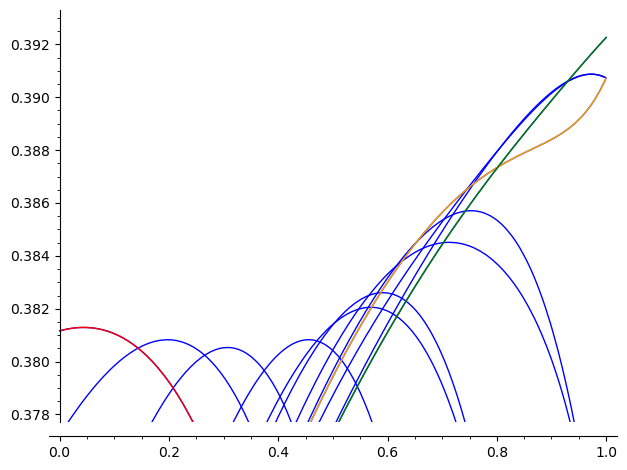}
\caption{Probability curves for (a) $N = 11$, $I = 6$ and (b) $N=24$, $I = 11$.}\label{f:pc}
\end{figure}

\begin{example}
    To be explicit, here are the data for the five curves shown in Figure~\ref{f:pc}(a).  We have only included strategies/diagrams that eventually become optimal for some value of $x$ between $0$ and $1$.

\[ \scalebox{0.5}{ $\begin{matrix}
        k \backslash j & 1 & 2 & 3 & 4 & 5 & 6 & 7 \\
        4 & - & - & - & - & - & - & - \\
        5 & + & + & + & + & + & + & + \\
\end{matrix}$ } \hspace{.2in} \scriptstyle \frac{251}{630} (x - 1)^6 - \frac{3013}{1260} (x - 1)^5 x + \frac{1121}{189} (x - 1)^4 x^2 - \frac{54}{7} (x - 1)^3 x^3 + \frac{38}{7} (x - 1)^2 x^4 - \frac{82}{45} (x - 1) x^5 + \frac{2}{15} x^6 \]

\[ \scalebox{0.5}{ $\begin{matrix}
        k \backslash j & 1 & 2 & 3 & 4 & 5 & 6 & 7 \\
        5 & - & - & - & - & - & - & - \\
        6 & + & + & + & + & + & + & + \\
        \end{matrix}$ } 
        \hspace{.2in} \scriptstyle 
        \frac{2131}{5544} (x - 1)^6 - \frac{2383}{1008} (x - 1)^5 x + \frac{4555}{756} (x - 1)^4 x^2 - \frac{905}{112} (x - 1)^3 x^3 + \frac{165}{28} (x - 1)^2 x^4 - \frac{149}{72} (x - 1) x^5 + \frac{1}{6} x^6 
\]

\[ \scalebox{0.5}{ $\begin{matrix}
        k \backslash j & 1 & 2 & 3 & 4 & 5 & 6 & 7 \\
        6 & - & - & - & - & - & - & - \\
        7 & \emptyset & \emptyset & \emptyset & \emptyset & \emptyset & \emptyset & + \\
        8 & + & + & + & + & + & + & + \\
\end{matrix}$ } 
        \hspace{.2in} \scriptstyle 
        \frac{1627}{4620} (x - 1)^6 - \frac{1879}{840} (x - 1)^5 x + \frac{743}{126} (x - 1)^4 x^2 - \frac{459}{56} (x - 1)^3 x^3 + \frac{87}{14} (x - 1)^2 x^4 - \frac{137}{60} (x - 1) x^5 + \frac{1}{5} x^6 
\]

\[ \scalebox{0.5}{ $\begin{matrix}
        k \backslash j & 1 & 2 & 3 & 4 & 5 & 6 & 7 \\
        6 & - & - & - & - & - & - & - \\
        7 & \emptyset & \emptyset & \emptyset & \emptyset & \emptyset & \emptyset & - \\
        8 & + & + & + & + & + & + & + \\
        \end{matrix}$ } \hspace{.2in} 
\scriptstyle
\frac{1207}{3960} (x - 1)^6 - \frac{55}{28} (x - 1)^5 x + \frac{2665}{504} (x - 1)^4 x^2 - \frac{319}{42} (x - 1)^3 x^3 + \frac{171}{28} (x - 1)^2 x^4 - \frac{77}{30} (x - 1) x^5 + \frac{5}{12} x^6 \]

\[ \scalebox{0.5}{ $\begin{matrix}
        k \backslash j & 1 & 2 & 3 & 4 & 5 & 6 & 7 \\
        6 & - & - & - & - & - & - & - \\
        7 & \emptyset & \emptyset & \emptyset & \emptyset & \emptyset & \emptyset & - \\
        8 & + & + & + & + & + & + & - \\
        9 & + & + & + & + & + & + & + \\
    \end{matrix}$ } \hspace{.2in} \scriptstyle
 \frac{821}{2772} (x - 1)^6 - \frac{9587}{5040} (x - 1)^5 x + \frac{1285}{252} (x - 1)^4 x^2 - \frac{2449}{336} (x - 1)^3 x^3 + \frac{35}{6} (x - 1)^2 x^4 - \frac{59}{24} (x - 1) x^5 + \frac{13}{30} x^6. \]
\end{example}

\bigskip
These results give us a fairly complete description of the solution to particular 1-step reselect games.  For future work, it would be interesting to resolve some conjectures that become apparent after computing several examples as we have done here.

\begin{itemize}
    \item We know that the optimal transition position $k$ is bounded as in Corollary~\ref{c:band}.  The ends of this band occur for the extreme values $x = 0$ and $x = 1$.  Therefore, it is natural to wonder whether the optimal transition positions $k$ are monotonically increasing in $x$?

    \item It is evident from the examples that the optimal win probability is {\em not} monotonically increasing in $x$.  Nevertheless, it appears that the win probability for the $x = 1$ model is larger than the optimal win probability of the game for any other value of $x$.  Any argument to attempt this must contend with phenomena such as the orange curve in Figure~\ref{f:pc}(b), which appears to be more complex than the other curves (e.g. containing an inflection point), whose data are
        \[ \hspace{.2in} \scalebox{0.5}{ $\begin{matrix}
        k \backslash j & 1 & 2 & 3 & 4 & 5 & 6 & 7 & 8 & 9 & 10 & 11 & 12 \\
        11 & - & - & - & - & - & - & - & - & - & - & - & - \\
        12 & + & + & + & + & + & - & - & - & - & - & - & - \\
        13 & + & + & + & + & + & + & + & - & - & - & - & - \\
        14 & + & + & + & + & + & + & + & + & + & - & - & - \\
        15 & + & + & + & + & + & + & + & + & + & + & + & - \\
        16 & + & + & + & + & + & + & + & + & + & + & + & + \\
        \end{matrix}$ } \]
    \[ \scriptstyle -\frac{1885269798611}{5397062711040} (x - 1)^{11} + \frac{4661371313}{1189828640} (x - 1)^{10} x - \frac{7429958203}{372468096} (x - 1)^9 x^2 + \frac{81121037923}{1333266480} (x - 1)^8 x^3 - \frac{8709570881}{70543200} (x - 1)^7 x^4 \]
    \[ \scriptstyle + \frac{1480843127}{8465184} (x - 1)^6 x^5 - \frac{2832259067}{16039296} (x - 1)^5 x^6 + \frac{11016021}{86632} (x - 1)^4 x^7 - \frac{8920823}{139776} (x - 1)^3 x^8 + \frac{7012567}{327600} (x - 1)^2 x^9 - \frac{7800211}{1834560} (x - 1) x^{10} + \frac{35201}{90090} x^{11}. \]

    \item It would be interesting to characterize the types of curves that can appear along the strategic boundary of our ``+/- diagrams.''  What do these diagrams look like for large $N$, with $I$ equal to a fixed proportion of $N$, say?
\end{itemize}

%%%%%%%%%%%%%%%%%%%%%%%%%%%%%%%%%%%%%%%%%%%%%%%%%%%%%%%%%%%%%%%%%%%%%
%  Acknowledgements
%%%%%%%%%%%%%%%%%%%%%%%%%%%%%%%%%%%%%%%%%%%%%%%%%%%%%%%%%%%%%%%%%%%%%
%\bigskip
\medskip
\section*{Acknowledgements}

This paper is dedicated to Edwin O'Shea, with whom the second author was
fortunate to share many mathematically inspiring conversations over eighteen
years of friendship.

\medskip
This project began at the 2025 Research Experiences for Undergraduate Faculty
(REUF) workshop at ICERM, supported by the National Science Foundation through
DMS grant 1620080.  We are grateful for contributions from Jonathan Brown, Erin
Griesenauer, and Molly Lynch, who participated in our initial working group in
2025.  We thank Alejandro Morales for helpful comments on a draft of this
manuscript.

%  Jonathan Brown (Bakersfield College), Erin Griesenauer (Eckerd College), Molly Lynch (Hollins University)

%%%%%%%%%%%%%%%%%%%%%%%%%%%%%%%%%%%%%%%%%%%%%%%%%%%%%%%%%%%%%%%%%%%%%
% ==========   Bibliography
%%%%%%%%%%%%%%%%%%%%%%%%%%%%%%%%%%%%%%%%%%%%%%%%%%%%%%%%%%%%%%%%%%%%%

\bigskip
\bibliographystyle{alpha}
\bibliography{bestchoice}

\end{document}